\documentclass[10pt]{amsart}
\usepackage{amssymb}
\usepackage{mathtools}

\usepackage{graphicx}

\usepackage{xcolor}
\usepackage{tikz-cd}

\newcommand{\res}{\upharpoonright}
\newcommand{\M}{{\mathbb M}}

\newcommand{\Fraisse}{Fra{\"i}ss{\'e}}

\usepackage{accents}
\newcommand{\orient}[1]{\accentset{\rightharpoonup}{#1}}
\theoremstyle{plain}
\newtheorem{theorem}{Theorem}[section]
\newtheorem{corollary}{Corollary}[section]
\newtheorem{lemma}{Lemma}[section]
\newtheorem{proposition}{Proposition}[section]
\newtheorem{definition}{Definition}[section]

\theoremstyle{remark}

\newtheorem*{claim}{Claim}

\theoremstyle{definition}
\newtheorem*{construction}{Construction of a generic sequence}
\newtheorem*{property}{Lifting property for $\mathbb{M}$}

\begin{document}

\title[A combinatorial model for the Menger curve]{A combinatorial model for the Menger curve}

\author{Aristotelis Panagiotopoulos}
\address{Department of Mathematics, Caltech, 1200 E. California Blvd, MC 253-37
Pasadena, CA 91125}
\email{panagio@caltech.edu}
\urladdr{http://www.its.caltech.edu/~panagio/}

\author{S\l awomir Solecki}
\address{Department of Mathematics, Malott Hall, Cornell University, Ithaca, NY 14853}
\email{ssolecki@cornell.edu}
\urladdr{https://math.cornell.edu/slawomir-solecki}

\thanks{Research of Solecki supported by NSF grants DMS-1800680 and 1700426.}


\begin{abstract} 
We represent the universal Menger curve as the topological realization $|\mathbb{M}|$ of the projective Fra{\"i}ss{\'e} limit ${\mathbb M}$ of 
the class of all finite connected graphs.
We show that $\mathbb{M}$ satisfies combinatorial analogues of the Mayer--Oversteegen--Tymchatyn homogeneity theorem and 
the Anderson--Wilson projective universality theorem. Our arguments involve only $0$-dimensional topology and constructions on finite graphs. 
Using the topological realization $\mathbb{M}\mapsto|\mathbb{M}|$, we transfer some of these properties to the Menger curve: we prove 
the approximate projective homogeneity theorem, 
recover Anderson's finite homogeneity theorem, and  prove a variant of Anderson--Wilson's theorem. The finite homogeneity theorem is 
the first instance of an ``injective" homogeneity theorem 
being proved using the projective Fra{\"i}ss{\'e} method. 
We indicate how our approach to the Menger curve may extend to higher dimensions.  
\end{abstract}

\maketitle

\section*{Introduction}

The Menger curve is a $1$-dimensional Peano continuum that is classically extracted from the cube in the same way  that the Cantor space is extracted from the interval: subdivide $C_0=[0,1]^3$ into $3^3$ congruent subcubes; let $C_1$ be the union of these subcubes which intersect the one-skeleton of $[0,1]^3$; repeat this process on each subcube again and again to define $C_n$; the Menger curve is defined to be the intersection $\bigcap_n{}C_n$.
With this construction Menger found the first example of a universal space for the class of $1$-dimensional continua, that is, 
a $1$-dimensional continuum in which every other $1$-dimensional continuum embeds \cite{Mg}. 

\begin{figure}[ht!]
\centering
\includegraphics[scale=0.3]{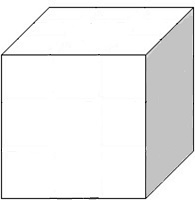} $\quad$ $\quad$ \includegraphics[scale=0.3]{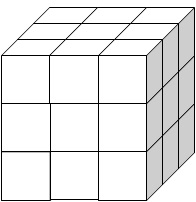} $\quad$ $\quad$ \includegraphics[scale=0.6]{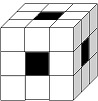}
\caption{From $C_0$ to $C_1$.}
\end{figure}

The Menger curve is a canonical continuum whose topological properties do not depend on the various geometric parameters of the 
above iterative process. In fact, many other constructions of universal $1$-dimensional continua (e.g.,  \cite{Lf,No}) which appeared soon after  \cite{Mg}  
were later shown to produce the same space; see \cite{An}.

In this paper, we develop a combinatorial model for the Menger curve using an analogue of projective \Fraisse{} theory from  \cite{IS}. The \emph{Menger prespace}  
$\mathbb{M}$ is a compact graph-structure on the Cantor space. In a sense,  $\mathbb{M}$ is the \emph{generic inverse limit} in the category  $\mathcal{C}$ of all \emph{connected epimorphisms} between \emph{finite connected graphs}. The edge relation on  $\mathbb{M}$ turns out to be an equivalence relation and the Menger curve is then defined to be the quotient $|\mathbb{M}|=\mathbb{M}/R$ of $\mathbb{M}$ with respect to this relation.

This definition of the Menger curve as the \emph{topological realization} $|\mathbb{M}|$  of the combinatorial object $\mathbb{M}$  has certain technical and foundational advantages. 
On the foundational side, the definition of $\mathbb{M}$ is  canonical since it is constructed through $\mathcal{C}$ without making any ad-hoc choices for the bonding maps. Moreover, the
definition of $|\mathbb{M}|$ is intrinsic,  in that it makes no reference to external spaces such as such as $[0,1]^3$. On the technical side, when proving results about the Menger curve, we can often replace various complications coming from $1$-dimensional topology of $|\mathbb{M}|$ with combinatorial problems about graphs in 
$\mathcal{C}$. 
Moreover, like any other projective \Fraisse{} limit, $\mathbb{M}$ has  the following \emph{projective extension property} built in by the construction:  for every $g\in\mathcal{C}$ and any connected epimorphism $f$ as in the diagram, there is a connected epimorphism $h$ with $g \circ h=f$.
\begin{center}
\medskip{}
\begin{tikzcd}[column sep=small]
\mathbb{M} \arrow[drrrr, dotted,  "h", swap] \arrow[rrrrrr, "f" ] & & & & & & A\\
& & & & B  \arrow[urr, "g", swap] & & 
\end{tikzcd}
\medskip{}
\end{center}
Having this universal property of $\mathbb{M}$ as our point of departure, and expanding on it using combinatorial properties of $\mathcal{C}$, 
we can integrate various aspects of the Menger curve into a unified theory as follows. 
\begin{itemize}

\item Anderson's homogeneity theorem \cite{An} states that any bijection between finite subsets of $|\mathbb{M}|$ extends to a global homeomorphism of $|\mathbb{M}|$. This theorem was later generalized in \cite{MOT} to the strongest possible homogeneity result for $|\mathbb{M}|$, namely, that every homeomorphism between locally non-separating, closed subsets of $|\mathbb{M}|$ extends to a global homeomorphism of $|\mathbb{M}|$. In Theorem \ref{T:2}, we prove a homogeneity result for ${\mathbb M}$ analogous to the homogeneity result for the Menger curve in \cite{MOT}. From that we recover Anderson's homogeneity result for $|\mathbb{M}|$. 
Our proof of Theorem \ref{T:2} relies on $\mathcal{C}$ being closed under a certain mapping cylinder construction. 

\item   Anderson--Wilson's projective universality theorem states that $|\mathbb{M}|$ admits an open, continuous, and connected map onto any Peano continuum\footnote{In this paper we use the newer term \emph{connected map} in place of the synonymous term \emph{monotone map} used in \cite{An2,Wi}.}. 
Moreover, all preimages of points under this map can be taken to be homeomorphic to the Menger curve  \cite{An2,Wi}. In Theorem \ref{T:preMengerProjUniversality}, we prove a combinatorial analogue of the Anderson--Wilson theorem for $\mathbb{M}$. In the process, we isolate a combinatorial property of $\mathcal{C}$ 
that is responsible for this strong form of projective universality. In Corollary \ref{C:MengerProjUniversality}, 
we  establish a variant of Anderson--Wilson's theorem for $|\mathbb{M}|$ where the map produced is weakly locally-connected instead of open. 

\item 
In Theorem \ref{T:ApproximateHomogeneityProperty}, we prove that $|\mathbb{M}|$ satisfies an approximate projective homogeneity property 
that is analogous to the property satisfied by many other continua presented as topological realizations of projective Fra{\"i}ss{\'e} limits; 
see \cite{BK} and \cite{IS} for examples. Namely, we show that 
if $\gamma_0,\gamma_1\colon |\mathbb{M}|\to X$ are continuous connected maps onto a Peano continuum $X$, 
then there is a sequence $(h_n)$ of homeomorphisms of  $|\mathbb{M}|$ so that $(\gamma_0\circ h_n)$ converges uniformly to $\gamma_1$.
\end{itemize}

It is worth mentioning that throughout Section \ref{S: Homogeneity} one can find analogies with the abstract homotopy theory in the spirit of model categories. 

Finally, pursuing an extension of this approach to higher dimensional Menger compacta, we define higher dimensional analogues of 
$\mathcal{C}$, $\mathbb{M}$, and 
$|\mathbb{M}|$. For every $n\in\{0,1,\ldots\}\cup\{\infty\}$, we define  $\mathcal{C}_n$ to be the class of all $(n-1)$-connected epimorphisms between finite, $n$-dimensional, 
$(n-1)$-connected simplicial complexes. We show that  $\mathcal{C}_n$ is a projective \Fraisse{} class. Interestingly, the same is shown to hold for the class 
$\widetilde{\mathcal{C}}_n$, which is defined by replacing ``$(n-1)$-connected'' with ``$(n-1)$-acyclic'' in the definition of $\mathcal{C}_n$. As far as we are aware, 
these ``homology $n$-Menger spaces'' introduced here---and for $n=\infty$ this ``homology Hilbert cube''---have not been considered before.

\tableofcontents

\section{The class $\mathcal{C}$ of finite connected graphs}\label{S:graphs}

Let $A$ be a set and let $R $ be any subset of  $A^2$. We say that $R$ is a {\bf reflexive} if $R(a,a)$ holds for all $a\in A$. We say that $R$ is {\bf symmetric} if for every $a,b\in A$ we have that  $R(a,b)$ implies $R(b,a)$. We finally say that  $R$ is {\bf transitive} if the conjunction of $R(a,b)$ and $R(b,c)$ implies $R(a,c)$.  By a {\bf graph} $(A,R^A)$,  simply denoted by $A$, we mean a set $A$ together with a specified subset $R^{A}$ of $A^2$ that is both reflexive and symmetric.
Notice that reflexivity makes our definition of a graph non-standard but it allows us to treat graphs as $1$-dimensional simplicial complexes. A {\bf clique} of a graph $(A,R^A)$ is any subset $C$ of $A$ with the property that for all $a,b\in C$ we have that $R^A(a,b)$.
A map $f\colon B\to A$ is a {\bf homomorphism} between graphs if it maps edges to edges, that is, if $R^{B}(b,b')$ implies $R^{B}(f(b),f(b'))$, for every $b,b'\in B$. A homomorphism $f$ is  an {\bf epimorphism} if it is moreover surjective on both vertices and edges. An isomorphism is an injective epimorphism. By a {\bf subgraph} of a graph we understand an induced subgraph.

We isolate a collection $\mathcal{C}$ of finite graphs together with special epimorphisms between them,  the point being, that various topological and dynamical properties of the Menger curve are reflections of combinatorial properties of $\mathcal{C}$. A subset $X$ of a finite graph $A$ is {\bf connected} if, for all non-empty $U_1, U_2\subseteq X$ with $X = U_1\cup U_2$, there exist $x_1\in U_1$ and $x_2\in U_2$ such that $R^A(x_1, x_2)$. A graph $A$ is {\bf connected} if the domain of $A$ is a connected subset. An epimorphism $f\colon B\to A$ is {\bf connected} if the preimage of each connected subset of $A$ is a connected subset of $B$.

\begin{definition} 
Let $\mathcal{C}$ be the category of all finite connected graphs with morphisms in ${\mathcal C}$ 
being connected epimorphisms. 
\end{definition}

Our first task is to establish that $\mathcal{C}$ is a projective \Fraisse{} class. Projective \Fraisse{} theory was developed in \cite{IS} in the more general setting of 
$\mathcal{L}$-structures. For the sake of perspective, we recall from \cite{IS} the \Fraisse{} class axioms in this more general setup. For the unfamiliar reader, we point out that a graph is just an example of an $\mathcal{L}$-structure where the language $\mathcal{L}$ consists of a binary relation symbol $R$. An important difference between the definition below and the one in \cite{IS} is that here, a \Fraisse{} class is allowed to consist of a strict subcollection of epimorphisms, e.g. only the epimorphisms which are connected.
   
Let $\mathcal{F}$ be a class of finite $\mathcal{L}$-structures with a fixed family of morphisms among the structures in $\mathcal{F}$. We assume that each morphism is an epimorphism with respect to $\mathcal L$. 
We say that $\mathcal F$ is a {\bf projective Fra{\"i}ss{\'e} class} if 
\begin{enumerate}
\item $\mathcal F$ is countable up to isomorphism, that is, any sub-collection of pairwise non-isomorphic structures of $\mathcal F$ is countable;
\item morphisms are closed under composition and each identity map is a morphism;
\item for $B, C\in {\mathcal F}$; 
there exist $D\in {\mathcal F}$ and morphisms $f\colon D\to B$ and $g\colon D\to C$; and
\item for every two morphisms $f\colon B\to A$ 
and $g\colon C\to A$, 
there exist morphisms $f'\colon D\to B$ and $g'\colon D\to C$ such that $f\circ f' = g\circ g'$. 
\end{enumerate}

We will refer to the last property above as the {\bf projective amalgamation property}.
We have the following theorem.
 
\begin{theorem}\label{T:1}
$\mathcal{C}$ is a projective Fra{\"i}ss{\'e} family.
\end{theorem}
\begin{proof} 
We check here that $\mathcal{C}$ satisfies the projective amalgamation property. The rest of the properties follow easily. Let $f \colon B\to A$ and $g \colon C\to A$ be connected epimorphisms and let $D$ be the subgraph of the product graph $B\times C$, induced on domain
\[
\{ (b,c) \in B\times C \colon f(b) = g(c)\}.
\]
Recall that in the product graph $B\times C$ there is an edge between $(b,c)$ and  $(b',c')$ if and only if $R^{B}(b,b')$ and $R^{C}(c,c')$. Let also $f'=p_B\colon D\to B$, $g'=p_C\colon D\to C$ be the canonical projections. By the definition of $B\times C$ it is immediate that $\pi_B,\pi_C$ are homomorphisms. 

We show that $p_B$ is a connected epimorphism. By symmetry, the same argument applies for $p_C$.
The fact that $g$ is surjective on vertexes implies that $p_B$ is surjective on vertexes since for every $b$ there is a $c_b$ with $f(b)=g(c_b)$, and hence there is $d=(b,c_b)$ with $\pi_B(d)=b$. By the same argument, and since $g$ is surjective on edges, we have that $p_B$ is surjective on edges as well. So $p_B$ is an epimorphism.  Moreover, since $g$ is connected, $g^{-1}(f(b))$ is connected for every $b\in B$. Hence the point fibers \[p_B^{-1}(b)=\{b\}\times g^{-1}(f(b))\] 
of $\pi_B$ are connected for every $b\in B$. The following general lemma implies therefore that $\pi_B$ is connected.
\end{proof}

\begin{lemma}\label{L:T1}
A function between two graphs of ${\mathcal C}$ is a connected epimorphism if and only if it is an 
epimorphism and preimages of points are connected.
\end{lemma}

\begin{proof} Only the direction $\Leftarrow$ needs to be checked. Let $f\colon B\to A$ be an 
epimorphism such that preimages of points are connected. It suffices to show that preimages of edges 
are connected. Let $b_1, b_2\in B$ be such that $R^A(f(b_1), f(b_2))$. Since $f$ is an epimorphism, there 
are $b_1', b_2'\in B$ that form an edge and are such that $f(b_1')=f(b_1)$ and $f(b_2')=f(b_2)$. Since the preimages 
of $f(b_1)$ and $f(b_2)$ are connected, there is a path connecting $b_1$ with $b_1'$ and $b_2$ with $b_2'$. Since 
$b_1'$ and $b_2'$ are connected by an edge, $b_1$ and $b_2$ are connected by a path, as required. 
\end{proof}

\section{Topological graphs and Peano continua}\label{S: Topological graphs and Peano continua}

We import some notions from \cite{IS} and we apply them here in the special case of graphs. A {\bf topological graph} $K$ is a graph $(K,R^{K})$,  whose domain $K$ is a $0$-dimensional, compact, metrizable topological space and $R^{K}$ is a closed subset of $K^2$. All types of morphisms we consider between topological graphs are assumed to be continuous. Moreover, we automatically view all finite graphs as topological structures endowed with the discrete topology.

We extend $\mathcal{C}$ to the class $\mathcal{C}^{\omega}$ of all topological graphs and epimorphisms which are ``approximable'' within $\mathcal{C}$.  A concrete description of $\mathcal{C}^{\omega}$ is given in Proposition \ref{P:char}. The rest of the paragraph defines $\mathcal{C}^{\omega}$ in abstract terms. Let $(K_n,f^n_m,\mathbb{N})$  be an inverse system of finite connected graphs with bonding maps  $f^n_m\colon K_n\to K_m$ from $\mathcal{C}$. It is easy to check that the inverse limit $K=\varprojlim (K_n,f^{n}_m)\in {\mathcal C}^\omega$ is a topological graph, where $(x_0,x_1,\ldots)$ is $R$-connected with $(y_0,y_1,\ldots)$ in $K$ if for every $n$, $x_n$ is $R$-connected with $y_n$ in $K_n$; see for example the proof of Proposition \ref{P:char}. We collect in ${\mathcal C}^\omega$ all topological graphs $K$ which are inverse limits of sequences with bonding maps from $\mathcal{C}$. 
Notice that every finite connected graph is in $\mathcal{C}^{\omega}$. If $A\in {\mathcal C}$ and $K=\varprojlim (K_n,f^{n}_m)\in {\mathcal C}^\omega$, then an epimorphism $h\colon K\to A$ is in ${\mathcal C}^\omega$ if and only if there exists $m$, and a morphism $h'\colon K_m\to A$ in ${\mathcal C}$, such that $h$ is the composition of $h'$ with the canonical projection $f_m$ from $K$ to $B_m$. For two topological graphs $K,L\in {\mathcal C}^\omega$, an epimorphism $h\colon L\to K$ is said to be in ${\mathcal C}^\omega$ 
if for each $A\in {\mathcal C}$ and each $g\colon K\to A$ in ${\mathcal C}^\omega$, the composition $g\circ h$ is in ${\mathcal C}^\omega$. Finally, an epimorphism $h\colon L \to K$ is an isomorphism if it is injective and both $h,h^{-1}$ are in ${\mathcal C}^\omega$. Notice that $h$ is an isomorphism
between $K=\varprojlim (K_n,f^{n}_m)\in {\mathcal C}^\omega$ and $L=\varprojlim (L_n,g^{n}_m)\in {\mathcal C}^\omega$ if and only if there is a sequence $(h_i)$ of morphisms in $\mathcal{C}$ and 
two strictly increasing sequences $(k_i)$ and $(l_i)$ of natural numbers such that for each $i$ 
\[
h_{2i}\circ h_{2i+1} = f^{k_{i+1}}_{k_{i}}\;\hbox{ and }\; h_{2i+1}\circ h_{2i+2} = g^{l_{i+1}}_{l_{i}}.
\]

We now give a more concrete description of the graphs and morphisms of ${\mathcal C}^\omega$. Let $K$ be a topological graph. We say a subset $X$ of $K$ is {\bf connected} if, for all open $U_1, U_2\subseteq K$ with $X\cap U_1\not=\emptyset \not= X\cap U_2$ 
and $X\subseteq U_1\cup U_2$, there exist $x_1\in X\cap U_1$ and $x_2\in X\cap U_2$ such that $R^K(x_1, x_2)$. We say that a topological graph $(K,R^{K})$ is {\bf connected} if $K$ is connected as a subset of the graph. We say that it is {\bf locally-connected} if it admits a basis of its topology consisting of connected sets in then above sense. Let $K,L$ be topological graphs and let $f\colon L \to K$ be an epimorphism. We say that $f$ is a {\bf connected  epimorphism} if the preimage of each closed connected subset of $K$ is connected. Note that the above notions coincide with the analogous notions introduced for finite graphs. 

\begin{proposition}\label{P:char}
$\mathcal{C}^\omega$ is the class of all connected epimorphisms between connected, locally-connected, topological graphs.
\end{proposition}
\begin{proof}

Let $K=\varprojlim (K_n,f^{n}_m)\in {\mathcal C}^\omega$ with $f^n_m\in\mathcal{C}$.
The underlying space of the graph $K$ is $0$-dimensional, compact, and  metrizable, since it is a countable inverse limit of discrete finite spaces. 
The set $R^{K}$ is closed and contains the diagonal as an intersection of closed relations containing the diagonal. This proves that that $K$ is a topological graph. We see now that $K$ is also connected.  Let also $f_n\colon K\to K_n$ be the projection induced by the inverse system. Assume that $U_1, U_2$ are non-empty open subsets of $K$ with  $K\subseteq U_1\cup U_2$. Since $K_0$ is connected, we can pick $x_0\in f_0(U_1)$ and $y_0\in f_0(U_2)$ with $R^{K_0}(x_0, y_0)$. Assume by induction that we picked $x_n \in f_n(U_1)$ and $y_n \in f_n(U_2)$, with $R^{K_n}(x_n, y_n)$, so that $f^n_{n-1}(x_n)=x_{n-1}$ and $f^n_{n-1}(y_n)=y_{n-1}$. Using the fact that $f^{-1}_n(\{x_n,y_n\})$ is connected and that $f^{n+1}_n$ is an epimorphism we can pick $x_{n+1}\in f_{n+1}(U_1)$ and $y_{n+1}\in f_{n+1}(U_2)$ with $R^{K_{n+1}}(x_{n+1}, y_{n+1})$, and so that  $f^{n+1}_n(x_{n+1})=x_{n}$ and $f^{n+1}_n(y_{n+1})=y_{n}$. Hence, $(x_0,x_1,\ldots)\in U_1$ and $(y_0,y_1,\ldots)\in U_2$ are such that $R^K((x_0,x_1,\ldots), (y_0,y_1,\ldots))$. The exact same argument can be applied to show that every clopen set of $K$ of the form $f^{-1}_n(x)$, where $x\in K_n$, is connected. Hence $K$ is locally-connected as well.

 Let now $L=\varprojlim (L_n,g^{n}_m)\in {\mathcal C}^\omega$ as well and let $h\colon L\to K$  be a morphism in $\mathcal{C}^\omega$. By definition, for every $m$ there is an $n$ and a connected epimorphism $h'\colon L_n \to K_m$, so that $h'\circ g_n=  f_m\circ h$, where $g_n\colon K\to K_n$ is the canonical projection. Since every connected clopen subset $\Delta$ of $K$ is of the form $f_m^{-1}(X)$ for large enough $m$ and some connected subset $X$ of $K_m$, we have that $h^{-1}(\Delta)= (h' \circ g_n)^{-1}(X)$ is a connected clopen subset of $L$. The rest follows from the fact that every closed connected subsets of $K$ and $L$ are interstions of a decreasing sequence of connected clopen subsets of the same spaces.

We turn to the converse statements first for graphs and then for morphisms. Let $K$ be a connected, locally-connected, topological graph. It is not difficult to see that $K$ admits a basis $\mathcal{U}$ of connected clopen sets. Using compactness of $K$ as well as of every element of $\mathcal{U}$, we can find a sequence $\mathcal{U}_n$ of finite covers of $K$ so that $\mathcal{U}_n\subset \mathcal{U}$, $\mathcal{U}_n$ refines $\mathcal{U}_{n-1}$, if  $U,V\in\mathcal{U}_n$ then $U\cap V=\emptyset$, and $\bigcup_n\mathcal{U}_n$ separates points of $K$. One can define a graph structure on $\mathcal{U}_n$ by putting an $R$-edge between $U$ and $V$ if there is $x\in U$ and $y\in V$ with $R^K(x,y)$. It is easy to see now that $f^{n}_m\colon \mathcal{U}_n\to \mathcal{U}_m$ is a connected epimorphism between finite connected graphs and that $K=\varprojlim (\mathcal{U}_n,f^{n}_m)$.

Let now $h\colon L\to K$ be a connected epimorphism between connected, locally-connected, topological graphs. By the previous paragraph  $K=\varprojlim (K_n,f^{n}_m)$ and $L=\varprojlim (L_n,g^{n}_m)$, where $f^n_m,g^n_m\in\mathcal{C}$. It suffices to show that for every $m$ there is $n$, and a map $h'\colon L_n\to K_m$ with $h'\in\mathcal{C}$ and $g_m \circ h= h' \circ f_n$. Fix some $m$ and let $n$ large enough so that $\{g^{-1}_n(y)\colon y\in L_n\}$ refines $\{(f_m\circ h)^{-1}(x)\colon x\in K_m\}$. Let also $h'\colon L_n\to K_m$ be the unique map that witnesses this refinement. Using that $f_m\circ h$ and $g_n$ are connected epimorphisms it is easy to see that $h'$ is in $\mathcal{C}$ as well. 
\end{proof}

Next we illustrate the relationship between topological graphs and Peano continua. Recall that a {\bf continuum} is a connected, compact,  metrizable space. A {\bf Peano continuum} is a continuum that is locally-connected. A map  $\phi\colon Y\to X$  between topological spaces is {\bf connected} if  $\phi^{-1}(Z)$ is connected for every closed connected subset $Z$ of $X$.
 Here connected and locally-connected refer to the standard topological notion. We also adopt the convention that the empty space is not connected.
   We will always accompany ambiguous terminology such as ``connected'' with further specification such as ``graph'' or ``space'' to distinguish between our combinatorial and the standard topological notion of connectedness.

A topological graph $K\in\mathcal{C}^{\omega}$ is a {\bf prespace} if the edge relation $R$ is also transitive. In other words, if $K$ is a collection of cliques. This makes $R$ an equivalence relation and we denote by $[x]$ the equivalence class of $x\in K$. Similarly, for every subset $F$ of $K$ we denote by $[F]$ the set of all $y\in K$ which lie in some equivalence class $[x]$ with $x\in F$. The {\bf topological realization} $|K|$ of a prespace $K$ is defined to be the quotient 
\[ K/R^K=\{[x]\colon x\in K\},\]  
endowed with the quotient topology. We denote by $\pi$ the quotient map $K\mapsto |K|$.
Since $R^K$ is compact, $|K|$ is compact and  metrizable. In fact, we have the following theorem.

\begin{theorem}\label{T: Peano <--> prespaces}
For a topological space $X$ the following are equivalent:
\begin{enumerate}
\item $X$ is a Peano continuum;
\item $X$ is homeomorphic to $|K|$ for some prespace $K\in\mathcal{C}^{\omega}$.
\end{enumerate}
\end{theorem} 

We start with a lemma.

\begin{lemma}\label{L: basis}
Let $K=\varprojlim(K_n,g^n_m)\in\mathcal{C}^{\omega}$ be a prespace, let $x\in K$ and let $g_n\colon K\to K_n$ be the natural canonical projections. Consider the following families:
\begin{itemize}
\item $\mathcal{P}^x_1=\{g^{-1}(a)\colon \; g\in\mathcal{C}^{\omega}, \; g([x])=a\}$, where $g$ ranges over all maps  $g\colon K\to A$ in $\mathcal{C}^{\omega}$, with $A\in \mathcal{C},$ and  $a\in A$; 
\item $\mathcal{P}^x_2=\{g^{-1}(Q)\colon \; g\in\mathcal{C}^{\omega}, \; g([x])=Q\}$, where $g$ ranges over all maps  $g\colon K\to A$ in $\mathcal{C}^{\omega}$, with $A\in \mathcal{C},$ and  $Q\subseteq A$;
\item $\mathcal{P}^x_3=\{g^{-1}(Q)\colon \; g\in\mathcal{C}^{\omega}, \; g([x])=Q\}$, where everything is as in $\mathcal{P}^x_2$, but $g$ ranges only over $\{g_n\colon n\in\mathbb{N}\}$.
\end{itemize}
If $\mathcal{P}^x$ is either of the above families, then $\mathcal{P}^x_{\pi}=\{\pi(P)\colon P\in \mathcal{P}^x\}$ is a neighborhood basis of $[x]$ in $|K|$ consisting of closed sets.
\end{lemma}
\begin{proof}
Let $P\in\mathcal{P}^{x}_i$ and set $U= [P^c]^c \subseteq K$. Notice that  $[P^c]$ is  the projection of the closed set
\[\{(x,y)\in K\times K\mid (x,y)\in \big(R^K\bigcap (K\times P^c)\big)  \},\]
along the compact second coordinate and therefore  $U$ is open. Since $R^K$ is an equivalence relation and $[P^c]$  is $R^K$ invariant, then  so is $U$. Hence, $\pi(U)$ is an open subset of $|K|$, and it clearly holds that  $[x]\in \pi(U)\subseteq \pi(P)$. 
 Since $\pi\colon K\to |K|$ is continuous and $P$ clopen we have that  $\pi(P)$ is a closed neighborhood of $[x]$. Compactness of $K$ implies that any open cover of $K$ can be refined by a partition of the form $\{g^{-1}_n(b)\colon b\in K_n\}$, for large enough $n$. Hence $\mathcal{P}^x_3$ projects through $\pi$ to a neighborhood basis of $[x]$. It is not difficult now to see that $\mathcal{P}^x_1=\mathcal{P}^x_2\supseteq \mathcal{P}^x_3$.
\end{proof}

We turn now back to the proof of Theorem \ref{T: Peano <--> prespaces}.

\begin{proof}[Proof of Theorem \ref{T: Peano <--> prespaces}]

First we show that $(2)\implies (1)$. Let $\mathcal{P}$ be the collection of clopen subsets of $K$ of the form $f^{-1}(a)$, where $f$ ranges over all $f\colon K\to A$ in $\mathcal{C}^{\omega}$ and $a\in A$. By Lemma \ref{L: basis}, $\mathcal{P}$ projects via $\pi$ to a neighborhood basis of $|K|$. It suffices to show that $\pi(P)$ is connected for every $P\in\mathcal{P}$; see Theorem 2.5 \cite{GM}, for example. Since every $P\in \mathcal{P}$ is itself an element of $\mathcal{C}^{\omega}$, it suffice to show that $|K|$ is a connected space for every prespace $K\in {\mathcal C}^\omega$. But any clopen partition of $|K|$ pulls back through $\pi$ to a clopen partition $\{U_1,U_2\}$ of $K$ which is invariant, that is, $[U_1]=U_1$ and $[U_2]=U_2$. By Theorem \ref{P:char}, $U_1$ is either empty or the whole space.

For  $(1)\implies (2)$, let $X$ be a Peano continuum. By Bing's Partition theorem (see \cite{Bi}) there is a sequence $(\mathcal{O}_n)$  of finite collections of disjoint open subsets of $X$, so that for all $n\in\mathbb{N}$ we have that:
\begin{enumerate}
\item $\bigcup\mathcal{O}_n$ is dense in $X$;
\item $O$ is connected, for all $O\in\mathcal{O}_n$;
\item $\mathcal{O}_{n+1}$ refines $\mathcal{O}_{n}$;
\item any open cover of $X$ is refined by $\mathcal{O}_{m}$, for large enough $m$.
\end{enumerate}
We turn each finite set $\mathcal{O}_n$ into a graph by putting an edge between $O $ and $O'$, if and only if, $\overline{O}\cap \overline{O'}\neq\emptyset$. Let $f^n_m\colon \mathcal{O}_n\to\mathcal{O}_m$ be the uniquely defined refinement map. Since every $O\in\bigcup_{n} \mathcal{O}_n$ is connected, it follows that $f^n_m\in\mathcal{C}$. Let $K=\varprojlim (\mathcal{O}_n,f^{n}_m)$. Let $\rho\colon K\to X$, mapping each point
 $x=(O_1,O_2,\ldots)\in K$ to the unique---by property $(4)$ above---point $\rho(x)$ with  $\{\rho(x)\}=\bigcap_{n}\overline{O_n}$. It is easy to see that $R^K$ is  the pullback of equality on $X$ under $\rho$, and hence, $K$ is a prespace with $X\simeq |K|$.
\end{proof}

\section{The Menger curve}\label{S: Menger curve}

The next theorem is proved using the methods of \cite{IS}. For completeness we summarize the construction of  ${\mathbb F}$ below. 
\begin{theorem}\label{T: characterization}
If $\mathcal F$ is a projective Fra{\"i}ss{\'e} family, then there exists a unique topological structure 
${\mathbb F}\in {\mathcal F}^\omega$ such that: 
\begin{enumerate}
\item for each $A\in {\mathcal F}$, there exists a morphism in ${\mathcal F}^\omega$ from $\mathbb F$ to $A$; 
\item for $A,B\in {\mathcal F}$ and morphisms $f\colon {\mathbb F}\to A$ and $g\colon B\to A$ in ${\mathcal F}^\omega$ 
there exists a morphism $h\colon {\mathbb F}\to B$  in ${\mathcal F}^\omega$ such that $f=g\circ h$.
\end{enumerate}
\end{theorem}
We say that $\mathbb F$ is the {\bf projective Fra{\"i}ss{\'e} limit of ${\mathcal F}$}. The second property in the above statement is called {\bf projective extension property}. 
We briefly sketch here the construction of $\mathbb F$ out of ${\mathcal F}$. 
For more details, see \cite{IS}.

\begin{construction}
We build $\mathbb F$ as an inverse limit of a \emph{generic sequence} $(L_n,t^n_m)$ of morphisms $t^n_m\in\mathcal{F}$. By property (1) in the definition of a \Fraisse{} class we can make two countable lists $(A_n\colon n\geq 0)$, $(e_n\colon C_n\to B_n\colon n\geq 1)$ containing all isomorphism types of structures and morphisms of $\mathcal{F}$. Moreover we make sure that every morphism type contained in $\mathcal{F}$ appears infinitely often in $(e_n)$ above.  Let $L_0=A_0$. Assume that $L_n$ has been defined together with all maps $t^n_i$, for all $i<n$. Using property (3) in the definition of a \Fraisse{} class we get $H\in\mathcal{F}$ together with maps $f\colon H\to L_n$, $g\colon H\to A_{n+1}$. Notice now that since $H$ is finite, there is a finite list $s_1,\ldots,s_k$ of morphism types  from $H$ to $B_{n+1}$ in $\mathcal{F}$. Using  $k$-many times projective amalgamation we get $f'\colon H'\to H$ and $d_j\colon H'\to C_{n+1}$ in $\mathcal{F}$ with $s_j\circ f'=e_{n+1}\circ d_j$ for all $j\leq k$. Set $L_{n+1}=H'$ and $t^{n+1}_i = t^{n}_i \circ f \circ f'$.   It is not difficult to see that the way ``saturated''  $(L_n,t^n_m)$ with respect to $(A_n)$ and $(e_n)$ endows $\mathbb{F}$ with properties (1) and (2) of Theorem \ref{T: characterization} above.
\end{construction}

As a consequence of Theorems ~\ref{T:1},\ref{T: characterization}, we can now consider  projective Fra{\"i}ss{\'e} limit ${\mathbb M}$ of $\mathcal{C}$. We call ${\mathbb M}$ the {\bf Menger prespace}.

\begin{theorem}\label{T: Menger is Menger}
The Menger prespace $\mathbb{M}$ is a prespace containing cliques of size at most $2$. Its topological realization $|\mathbb{M}|$ is the Menger curve.
\end{theorem}
\begin{proof}
The Menger curve is  the unique $1$-dimensional, Peano continuum
with the disjoint arcs property (\cite{Be}, see also \cite{An,MOT}). Recall that  a space $X$ has the disjoint arcs property if every continuous map $ \{0,1\}\times [0,1] \mapsto X$ can be uniformly approximated by maps which send $\{0\}\times[0,1]$ and $\{1\}\times[0,1]$ to disjoint sets. 

By Theorem \ref{T: Peano <--> prespaces}, we know that $|{\mathbb M}|$ is a Peano continuum.  To show that $|\mathbb{M}|$ is $1$-dimensional we find for every open cover a refinement whose nerve is one-dimensional. Let $\mathcal{V}$ be any open cover of $|\mathbb{M}|$ and let $f\colon \mathbb{M}\to A$ be any $f\in\mathcal{C}^{\omega}$ with $A\in\mathcal{C}$, so that $\mathcal{V}_f=\{\pi(f^{-1}(a))\colon a\in A\}$ refines $\mathcal{V}$. Let  $g\colon B \to A$ be in $\mathcal{C}$, so that $B$ has no cliques of size $3$. For example one can barycentrically subdivide $A$ and map the new vertexes to either of its two neighbors. The projective extension property of $\mathbb{M}$ provides us with a further refinement $\mathcal{V}_{h}=\{\pi(h^{-1}(b))\colon b\in B\}$ of $\mathcal{V}_f$. Notice that since $B$ has no cliques of size $3$, the nerve of $\mathcal{V}_{h}$ is isomorphic to $B$. Since $|\mathbb{M}|$ is a regular topological space and $\mathcal{V}_{h}$ is finite, we can find for every $W\in\mathcal{V}_{h}$ an open $U_W\supseteq W$, so that $\{U_W\colon W\in\mathcal{V}_{h}\}$ has the same nerve as $\mathcal{V}_{h}$ and still refines $\mathcal{V}$.

For the disjoint arcs property, let $\gamma_0,\gamma_1\colon [0,1]\to |\mathbb{M}|$ be two maps and let $\mathcal{V}$ be an open cover of $|\mathbb{M}|$. We will find disjoint $\gamma'_0,\gamma'_1\colon [0,1]\to |\mathbb{M}|$ which are $\mathcal{V}$-close to $\gamma_0$ and $\gamma_1$, that is,  for every $x\in[0,1]$, there is $V\in\mathcal{V}$, so that both  $\gamma_i(x),\gamma'_i(x)$ lie in $V$. As in the previous paragraph, let $\mathcal{V}_f=\{\pi(f^{-1}(a))\colon a\in A\}$ be a refinement of $\mathcal{V}$ and consider an open cover $\mathcal{U}_f=\{U_a\colon a\in A\}$ refining of $\mathcal{V}$, with $U_a\supseteq\pi(f^{-1}(a))$, having the same nerve as $\mathcal{V}_f$. Notice that for every $i\in\{0,1\}$ there is a finite cover $\mathcal{V}^i$ of $[0,1]$ with connected open intervals, and an assignment $\alpha_i\colon\mathcal{V}^i\to A$, so that $\gamma_i(V)\subseteq U_a$, for every $V\in\mathcal{V}^i$ with $\alpha_i(V)=a$.   Let $J$ be the unique graph on domain $\{0,\frac{1}{2},1\}$ so that $R^J(j,j')$ if and only if $|j-j'|\leq \frac{1}{2}$, and notice that the canonical projection $\rho\colon J\times A\to A$ is in $\mathcal{C}$. Hence by the projective extension property of $\mathbb{M}$ we have a connected epimorphism $h\colon\mathbb{M}\to J\times A$ so that $f=\rho\circ h$. Using the fact that $\pi(h^{-1}(X))$ is path-connected for every connected subset $X$ of $J\times A$, it is easy now to construct a  
paths $\gamma'_0$ and $\gamma'_1$ which are $\mathcal{V}$-close to the original paths and that moreover, $\gamma_i([0,1])\subset  \pi( h^{-1}(\{i\}\times A))$. 
\end{proof}

\section{The combinatorics of homogeneity} \label{S: Homogeneity}

In Theorem~\ref{T:2} below, we prove an injective homogeneity result for $\mathbb{M}$ analogous to the main result for $|\mathbb{M}|$ in \cite{MOT}. 
In Corollary \ref{C:hmg}, we recover Anderson's homogeneity result for the Menger curve $|\mathbb{M}|$. We note that, as in Section \ref{S:ApproximateProjectiveHomogeneity}, an appropriate version of {\em projective} homogeneity can always be obtained naturally and without much difficulty for any continuum which has been  presented as a topological realization of some projective Fra{\"i}ss{\'e} limit; see \cite{BK} and \cite{IS} for example.
 Here we provide the first example of an  {\em injective} homogeneity  property that is  obtained using projective  Fra{\"i}ss{\'e} theoretic methods.

Let $K$ be a closed subgraph of $\mathbb{M}$. We say that $K$ is {\bf locally non-separating} if for each clopen connected $W$, the set $W\setminus K$ is connected. 
\begin{theorem}\label{T:2}
If $K=[K]$ and $L=[L]$ are locally non-separating subgraphs of $\M$, then each isomorphism from $K$ to $L$ extends to an automorphism of $\mathbb{M}$.
\end{theorem}

For the proof of Theorem~\ref{T:2} we run a standard ``back and forth'' argument based on a lifting property for inclusions $K\hookrightarrow\mathbb{M}$ of locally non-separating sets; see page \pageref{Lifting property}. 
This lifting property strengthens the  projective extension property of $\mathbb{M}$. 

Viewed from an abstract homotopy theoretic standpoint, the lifting property 
suggests that maps in $\mathcal{C}$ relate to the above inclusion $K\hookrightarrow\mathbb{M}$ in the same way that trivial fibrations relate to  cofibrations within a model category.  The analogy with model categories is also reflected in the way we prove the lifting property: we define a combinatorial analogue of the \emph{mapping cylinder construction} for homomorphisms between graphs and we show that for any $f\colon\mathbb{M}\to A$ in $\mathcal{C}^{\omega}$, the induced map from $K$ to $A$ factors through a map of the form $r\circ i$, where $i$ is an inclusion and $r$ a mapping cylinder retraction. Before we describe the mapping cylinder construction we start with two general lemmas. The next result is probably known, but we could not find a reference for it.

\begin{lemma}\label{L:0} A closed subset $K$ of $\M$ is locally non-separating if and only if for each clopen connected set $W\supseteq K$ and 
each clopen set $V$ with $K\subseteq V\subseteq W$ there exists a clopen set $U$ such that $K\subseteq U\subseteq V$ and $W\setminus U$ is 
connected. 
\end{lemma}

\begin{proof} Only the direction from left to right needs a proof. 
Fix a connected clopen set $W$. Since $W\setminus K$ is open, we have $W\setminus K = \bigcup_{k\in {\mathbb N}} V_k$ for 
some $V_k$ clopen and connected. Let $k(0)=0$ and define $U_0=V_0$. Given $U_n$, let $k(n+1)$ be the smallest natural number such that 
$V_{k(n+1)}\not\subseteq U_n$ and $U_n\cup V_{k(n+1)}$ is connected, if such $k(n+1)$ exists. Otherwise, let $k(n+1)=k(n)$. Let $U_{n+1} = U_n\cup V_{k(n+1)}$. 

Since $U_n\subseteq U_{n+1}$ for each $n$, by compactness, it will suffice to show that 
\begin{equation}\label{E:ooo}
W\setminus K = \bigcup_{n\in {\mathbb N}}U_n. 
\end{equation}
This follows as in the last part of the proof of Lemma \ref{L:ApproximateHomogeneityProperty}: assume that
 $x\in W\setminus K$ and $x\not\in \bigcup_{n\in {\mathbb N}} U_n$; let $k(x)$ be such that  $x\in V_{k(x)}$; check that 
$[V_{k(x)}]\cap \bigcup_{n\in {\mathbb N}} U_n =\emptyset$; and derive a contradiction from the fact that $W\setminus K$ is connected.
%To see this equality, 
%let $x\in W\setminus K$ and $x\not\in \bigcup_{n\in {\mathbb N}} U_n$. Let $k(x)$ be such that 
%\begin{equation}\label{E:ppp}
%x\in V_{k(x)}. 
%\end{equation}
%Note that 
%\begin{equation}\label{E:qqq}
%[V_{k(x)}]\cap \bigcup_{n\in {\mathbb N}} U_n =\emptyset.
%\end{equation}
%Otherwise, $[V_{k(x)}]\cap U_{n_0}\not=\emptyset$ for some $n_0$, and it follows that, for each $n>n_0$, $k(n)\not= k(n+1)$ and 
%$k(n)<k(x)$, which is contradictory. It follows from \eqref{E:ppp} and \eqref{E:qqq} that for some $A\subseteq {\mathbb N}$, we have 
%\[
%(W\setminus K)\setminus \bigcup_{n\in {\mathbb N}} U_n = \bigcup_{k\in A}V_k\;\hbox{ and }\; [\bigcup_{k\in A}V_k] \cap \bigcup_{n\in {\mathbb N}} U_n=\emptyset. 
%\]
%Since $W\setminus K$ is connected, we get that $A=\emptyset$ and \eqref{E:ooo} follows. 
\end{proof}

\begin{lemma}\label{L: separation}
If $K\in\mathcal{C}^{\omega}$ is a prespace,  $Z\subseteq V \subseteq  K$,  $Z=[Z]$, and $V$ is open, then there is $W\subseteq K$ open with $Z\subseteq W$  and $[W]\subseteq V$. If moreover $Z$ is closed, then $W$ can be additionally chosen to  clopen.
\end{lemma}
\begin{proof}
It suffices to show that for every $z\in Z$ we can find $W_z$ clopen with $z\in W_z$ and $[W_z]\subseteq V$. If such $W_z$ doesn't exist then one can find sequences $(x_n)$and $(y_n)$ so that $y_n\in[x_n]$, $x_n$ converging to $z$, and $y_n\in V^c$. By compactness of $V^c$ we can assume that $(y_n)$ converges to $y\in Y$. But since $R^K$ is closed this implies that $y\in [x]$, contradicting that $Z=[Z]\subseteq V$.  
\end{proof}

Let $X$ be any finite (reflexive) graph and let $\alpha\colon X\to A$ be a graph homomorphism with $A\in \mathcal{C}$. We assume that $\mathrm{dom}(A)\cap\mathrm{dom}(X)=\emptyset$.
The {\bf mapping cylinder} $C_{\alpha}$ of $\alpha$ is the unique graph on domain $\mathrm{dom}(A)\cup \mathrm{dom}(X)$ with:
\begin{enumerate}
\item $C_{\alpha}\res \mathrm{dom}(A)=A$ and $C_{\alpha}\res \mathrm{dom}(X)=X$; 
\item for each $x\in X$ and $a\in A$, there is an edge in $C_{\alpha}$ between $x$ and $a$ if and only if $a=\alpha(x')$ for some $x'\in X$ with $R^X(x,x')$.
\end{enumerate}
The mapping cylinder $C_{\alpha}$ comes together with two natural graph inclusions $A,X\hookrightarrow C_{\alpha}$ and a {\bf canonical retraction} $r_{\alpha}\colon C_{\alpha}\to A$ given by: $r_{\alpha}(x)=\alpha(x)$, if $x\in X$; and $r_{\alpha}(x)=x$, otherwise. It is easy to check that both $C_{\alpha},r_{\alpha}$ are in $\mathcal{C}$.

\begin{lemma}\label{L: Mapping cylinder}
Let $K=[K]$ be a locally non-separating subgraph of $\mathbb{M}$; let $X$ be a finite graph; let $\alpha\colon X\to A$ be graph homomorphism, with $A\in\mathcal{C}$. For every $f\colon\mathbb{M}\to A$ in $\mathcal{C}^{\omega}$ and every graph homomorphism $q\colon K\to X$ with $\alpha\circ q=f\res K$, there is $\tilde{f}\colon\mathbb{M}\to C_{\alpha}$ in $\mathcal{C}^{\omega}$, with $r_{\alpha}\circ \tilde{f}=f$ and $\tilde{f}\res K= q$. 
\begin{center}
\begin{tikzcd}[column sep=small]
K \arrow[dd,hook']\arrow[rrrr, "q"] & & &  & X \arrow[dd,"\alpha"] \arrow[dll, hook'] \\
   & & C_{\alpha}\arrow[drr, "r_{\alpha}"] & & \\
\mathbb{M} \arrow[rrrr, "f"]\arrow[urr, "\tilde{f}",dotted] & & & &  A
\end{tikzcd}
\end{center}
\end{lemma}
\begin{proof} 

Let $K$, $X$,  $A$,  $\alpha$, $f$, $q$ be the data provided in the statement of Lemma \ref{L: Mapping cylinder} and set $K_x=q^{-1}(x)$, for every $x\in X$.

\begin{claim}
There is $g\colon B\to A$ in $\mathcal{C}$ and $h\colon \mathbb{M} \to B$ in $\mathcal{C}^{\omega}$, with $g\circ h=f$, together with collections $\{D_x\colon x\in X\}$ and $\{D_a\colon a\in A\}$ of subgraphs of $B$, so that if we set $B_a=g^{-1}(a)$ for all $a\in A$, we have:
\begin{enumerate}
\item $\{\mathrm{dom}(D_a)\}\cup\{\mathrm{dom}(D_x)\colon x\in X, \; \alpha(x)=a \}$ is a partition of $\mathrm{dom}(B_a)$; 
\item the image of $K_x$ under $h$ is contained in $D_x$;
\item $R^{X}(x,x')$ if and only if there is $b\in D_x$ and $b'\in D_{x'}$ with $R^B(b,b')$;
\item $R^{X}(x,x')$ for some $x'$ with $\alpha(x')=a$  if and only if there is $b\in D_x$ and $b'\in B_a$ with $R^B(b,b')$;
\item for every connected component $C$ of $D_x$ there is $c\in C$ and $b\in D_{\alpha(x)}$ with $R(b,c)$;
\item if $\widetilde{D}$ is the subgraph of $B$ on domain $\bigcup_{a\in A}\mathrm{dom}(D_a)$, then $g\res\widetilde{D}\colon \widetilde{D}\to A$ is in $\mathcal{C}$ and as a consequence $D_a$ is connected.
\end{enumerate}
\end{claim}

\begin{proof}[Proof of Claim.]
Since $\{K_x\colon x\in X\}$ is a finite collection of closed subsets of a $0$-dimensional, metrizable topological space we can find for each $x$ a clopen subset $W^0_x$ of $\mathbb{M}$ containing $K_x$ so that $W^{0}_x\cap W^{0}_x\neq\emptyset$ if and only if $x=x'$. By Lemma \ref{L: separation} we can find for each $x$ a clopen subset $W^1_x$ of $\mathbb{M}$ containing $[K_x]$ so that $[W^1_x]\cap[W^1_{x'}]\neq\emptyset$ if and only if $R^{X}(x,x')$ and $[W^1_x]\cap f^{-1}(a)\neq\emptyset$ if and only if there is $x'\in X$ with $R^{X}(x,x')$ and $\alpha(x')=a$. Finally, since $K$ is locally non-separating, we can chose for every---possibly trivial---edge $e=\{a,a'\}$ of $A$, a clopen subset  $W_e$ of $\mathbb{M}\setminus K$ so that $f(W_e)=e$.

Let now $h'\colon \mathbb{M}\to B'$ be any map in $\mathcal{C}^{\omega}$ which refines $f$ as well as the partition generated by all the sets $W^{0}_x$, $W^{1}_x$, $W_e$ collected above. Let also $g'\colon B'\to A$ be the unique map with $h'\circ g'= f$ and set $D'_x$ be the subgraph of $B'_a$ generated on domain $h'(K_x)$ and $D'_a$ be the subgraph of $B'_a$ generated on $\mathrm{dom}(B'_a)\setminus(\bigcup_x \mathrm{dom}(D'_x))$. It is easy to see that $g'\in \mathcal{C}$ and the resulting $h',g',\{D'_a\}, \{D'_x\}$ satisfy properties (1), (2), (3), (4) above. Moreover if $\widetilde{D}'$ is the subgraph of $B'$ on domain $\bigcup_{a\in A}\mathrm{dom}(D'_a)$, then $g'\res\widetilde{D}'\colon \widetilde{D}'\to A$ is an epimorphism.

Since locally non-separating sets are nowhere dense, for every $a\in A$ we can chose a clopen set $V'_a$ of $f^{-1}(a)\setminus K$ so that $h'(V'_a)$ intersects every connected component of every graph $D'_x$ with $a=\alpha(x)$. For every $a\in A$, set $W_a=f^{-1}(a)$, $K_a=W_a\cap K$, $V_a=\big((h')^{-1}(\bigcup_{x\colon \alpha(x)=a}D'_x)\big) \setminus V'_a$. By Lemma \ref{L:0} we get a clopen subset $U_a$ of $\mathbb{M}$ with $K_a\subseteq U_a\subseteq V_a$ so that $W_a\setminus U_a$ is connected. As above we can find $h\colon \mathbb{M}\to B$ in $\mathcal{C}^{\omega}$ and $g''\colon B\to B'$ in $\mathcal{C}$ with $g''\circ h= h'$, and so that $h$ refines the partition generated by $\{U_a\colon a\in A\}$. Set $g=g'\circ g''$ and $B_a=g^{-1}(a)$. Let also $D_a$ be the subgraph of $B_a$ on domain $h(W_a\setminus U_a)$ and let $D_x$ be the subgraph of $B_a$ on domain $(g'')^{-1}(D'_x)\cap h(U_a)$. Notice that all properties we established for $g'$ are preserved under refinements and that $g$ additionally satisfies properties (5) and (6).
\end{proof}

Given the configuration of the above claim, let $E_x$ be the subgraph of $B$ generated on $\mathrm{dom}(D_x)\cup \mathrm{dom}(D_{\alpha(x)})$. Properties (5) and (6) above imply that $E_x$ is connected. Let $E'_x$ be an isomorphic copy of $E_x$ and let $i_x\colon E'_x\to B$ be an embedding witnessing this isomorphism. Let $G$ be the mapping cylinder with respect to the map $i\colon \sqcup_{x\in X} E'_x\to B$, where $i=\sqcup_{x\in X}i_x$, and let $r\colon G\to B$ be the associated retraction.  By projective extension property we get $f_0\colon \mathbb{M}\to G$ with $(g \circ  r) \circ f_0=f$. By properties (3), (4), (5), (6) above, and the fact that $(f_1)^{-1}(x)$ is connected for all $x\in X$, we have the map $f_1\colon G\to C_{\alpha}$ that  maps $E'_x\cup D_x$ to $x$ and $D_a$ to $a$ is in $\mathcal{C}$. It is also immediate that $r_{\alpha}\circ f_1= g\circ r $. To finish the proof we set $\tilde{f}=f_1\circ f_0$. As a consequence we have $ r_{\alpha}\circ\tilde{f}=r_{\alpha}\circ f_1\circ f_0= g\circ r \circ f_0=f$ and $\tilde{f}(K_x)=f_1\circ f_0(K_x)\subseteq f_1(\mathrm{dom}(E'_x)\cup \mathrm{dom}(D_x))=\{x\}$.
\end{proof}

We can turn now to the proof of the main theorem of this section.

\begin{proof}[Proof of Theorem \ref{T:2}]
The proof of Theorem \ref{T:2} is a standard ``back and forth'' argument based on the following lifting property. Notice that the content of the lower commuting triangle is our usual projective extension property.

\begin{property}\label{Lifting property}
Let $K=[K]$ be a locally non-separating subgraph of $\mathbb{M}$. Let also $g\colon B\to A$ in $\mathcal{C}$ and $f\colon \mathbb{M}\to A$ in $\mathcal{C}^{\omega}$. Then for every graph homomorphism $p\colon K\to B$, with $g\circ p=f\res K$, there is $h\colon\mathbb{M}\to B$ in $\mathcal{C}^{\omega}$ with $g\circ h=f$ and $h\res K = p$. 
\begin{center}
\begin{tikzcd}[column sep=large, row sep=large]
K \arrow[r,  "p"] \arrow[d, hook] & B \arrow[d, "g" ]\\
\mathbb{M}  \arrow[ur, "h", dotted ,swap]  \arrow[r, "f",swap] & A 
\end{tikzcd}
\end{center}
\end{property}
We are left to show that the above lifting property holds. Notice first that if $g\colon B\to A$ is in $\mathcal{C}$ and $\beta\colon X\to B$, $\alpha\colon X\to A$, are graph homomorphisms with $g\circ\beta=\alpha$, then there is a unique extension $g^{*}\colon C_{\beta}\to C_{\alpha}$ of $g$ which makes the right diagram below commute. It is easy to check that $g^{*}$ is in $\mathcal{C}$.
\begin{center}
\begin{tikzcd}[column sep=small]
& & X \arrow[dll, "\beta",swap] \arrow[drr, "\alpha"] & & \\
B \arrow[rrrr, "g"] & & & & A
\end{tikzcd}
\quad\quad $\rightsquigarrow$ \quad\quad
\begin{tikzcd}[column sep=small]
& & X \arrow[dll, hook'] \arrow[drr, hook] & & \\
C_{\beta} \arrow[d, "r_{\beta}",swap] \arrow[rrrr, "g^*"] & & &  & C_{\alpha} \arrow[d,"r_{\alpha}",swap]   \\
B \arrow[rrrr, "g"] & & & &  A
\end{tikzcd}
\end{center}

Let now $f,g,p$ be as in the statement of the Lifting Property for $\mathbb{M}$ and let $X$ be a graph isomorphic to the graph that is the image of $K$ in $B$ under $p$. Let also $\beta\colon X\to B$ be this isomorphism and let $q\colon K\to X$ be the unique map with $\beta\circ q=p$. Notice that $\beta$ is not an embedding---in general---but it is always an injective homomorphism. Set $\alpha\colon X\to A$ be the homomorphism $g\circ \beta$. 

By Lemma \ref{L: Mapping cylinder} we have $\tilde{f}\colon\mathbb{M}\to C_{\alpha}$ in $\mathcal{C}^{\omega}$, with $r_{\alpha}\circ \tilde{f}=f$ and $\tilde{f}\res K=q$. Let $g^{*}\colon C_{\beta}\to C_{\alpha}$ be the extension of $g$ to $C_{\beta}$ described above. By the projective extension property of $\mathbb{M}$ we get a map $\tilde{h}\colon\mathbb{M}\to C_{\beta}$ with $g^{*}\circ\tilde{h}=\tilde{f}$. It follows that the map $h\colon\mathbb{M}\to B$ defined by $r_{\beta}\circ\tilde{h}$ is the desired map. To see this notice that $f=r_{\alpha}\circ \tilde{f}=r_{\alpha}\circ g^{*}\circ \tilde{h}=g\circ r_{\beta}\circ\tilde{h}=g\circ h$. By a similar diagram chasing, using that $g^{*}\res X=\mathrm{id}_X$ and $(g^{*})^{-1}(X)=X$ we get that $p=h\res K$.
\end{proof}

We finish this section by showing how one can derive Anderson's homogeneity for the Menger curve \cite{An} from Theorem~\ref{T:2}.

\begin{corollary}[Anderson~\cite{An}]\label{C:hmg}
Any bijection between finite subsets of $|\mathbb{M}|$ extends to a homeomorphism of $|\mathbb{M}|$. 
\end{corollary}
\begin{proof}

Let $\phi_0\colon F\to F'$ be a bijection between finite subsets of $|\mathbb{M}|$. If $\phi_0$ lifts through $\pi\colon \mathbb{M}\to |\mathbb{M}|$ to a bijection $\phi^{\pi}_0$ between $\cup F$ and  $\cup F'$ then by Theorem \ref{T:2},  $\phi^{\pi}_0$ extends to a global automorphism $\phi^{\pi}\colon\mathbb{M}\to\mathbb{M}$, and $\phi=\pi\circ\phi^{\pi}\circ\pi^{-1}$ is the required homeomorphism extending $\phi_0$. 
Here we used that finite subsets of $\mathbb{M}$ are locally non-separating, 
which easily follows from Lemma \ref{L:0} and the projective extension property of $\mathbb{M}$.  Hence the proof reduces to the following claim. 

\begin{claim}
For every finite subset $F$ of $|\mathbb{M}|$, there exists a homeomorphism $\psi\colon|\mathbb{M}|\to|\mathbb{M}|$ so that every element $[y]$ in $\psi(F)=\{\psi([x])\colon [x]\in F\}$ is a singleton (as a subset of $\mathbb{M}$).
\end{claim}

\noindent{\em Proof of Claim.}
Let $E$ be the equivalence relation on $\mathbb{M}$ defined by $x E x'$ if either $x=x'$; or if $x'\in [x]$ and $[x]\in F$. Let $\mathbb{M}'=\mathbb{M}/E$, let $\rho\colon \mathbb{M}\to \mathbb{M}'$ be the quotient map, and let $R'$ be the equivalence relation on $\mathbb{M}'$, that is the push-forward of $R$ under $\rho$. Since $\rho$ is $R$-invariant, $R'$ is well defined. Notice that $\rho$ is continuous since $E$ is compact and hence the induced map $|\rho|\colon |\mathbb{M}|\to \mathbb{M}'/R'$ on the quotients is a homeomorphism.  It suffices to show that there exists an isomorphism $\phi\colon\mathbb{M}\to \mathbb{M}'$ in $\mathcal{C}^{\omega}$.  If so, the map $ \pi\circ\phi^{-1}\circ (\pi')^{-1} \circ |\rho|$, where $\pi'\colon\mathbb{M}'\to \mathbb{M}'/R'$ is the quotient map, is the desired homeomorphism $\psi$. Hence, by Theorem \ref{T: characterization}, we have to check that $\mathbb{M}'$ (with the relation $R'$) is in $\mathcal{C}^{\omega}$ and that it satisfies properties (1) and (2) therein.

To see that $\mathbb{M}'$ is in $\mathcal{C}^{\omega}$ notice first that the union of any two $R$-connected clopen subsets of $\mathbb{M}$ is clopen and $R$-connected. Since $F$ is finite one can easily generate a basis for the topology of $\mathbb{M}'$ consisting of clopen $R'$-connected sets.  The rest follows from Proposition~\ref{P:char}. 

We now check that $\mathbb{M}'$ satisfies property (1) from Theorem \ref{T: characterization}. Let $A\in \mathcal{C}$ and let $n$ be a number strictly larger than the cardinality of $F$. Consider the graph $\delta^nA$ which is attained by subdividing every edge of $A$ $n$-times, that is, each non-trivial edge $(v,v')$ of $A$ is replaced a chain $(v,v_1), (v_1,v_2), \ldots, (v_n,v')$ of $n$-many edges. Notice that for every map  $(v,v')\mapsto_{\gamma}\{0,\ldots n\}$ which assigns to each edge $(v,v')$ of $A$ a number less or equal to $n$ we define a map $d_{\gamma}\colon \delta^nA\to A $ collapsing every vertex $v_m$ with $m>\gamma((v,v'))$ to $v'$ and  every vertex $v_m$ with $m\leq\gamma((v,v'))$ to $v$.  Let $f\colon \mathbb{M}\to \delta^nA $ be any $\mathcal{C}^{\omega}$ map. By the choice of $n$, there is an assignment $\gamma$ as above so that for every edge $(v,v')$ there is no $[x]\in F$ with $f([x])=(v_k,v_{k+1})$, where $k=\gamma((v,v'))$. The map $g\colon \mathbb{M}\to A$ with $g=d_{\gamma}\circ f$ is easily shown to push forward through $\rho$ to a $\mathcal{C}^{\omega}$ map $g^{\rho}\colon \mathbb{M}'\to A$. 

Property (2) from Theorem \ref{T: characterization} is proved for $\mathbb{M}'$ in a similar fashion. 
Let $f\colon \mathbb{M}'\to A$ in $\mathcal{C}^{\omega}$ and $g\colon B\to A$ in $\mathcal{C}$. Notice that $f\colon \rho \colon \mathbb{M}\to A$ is in $\mathcal{C}^{\omega}$. We can now construct the desired map $h\colon \mathbb{M}'\to B$ by relativizing the argument of the previous paragraph with respect to the constrains $f$ and $g$. 
The claim and, therefore, also the corollary follow. \end{proof}

\section{The combinatorics of universality} \label{S: Universality}

In Theorem \ref{T:preMengerProjUniversality} we prove for $\mathbb{M}$ a combinatorial analogue of a strengthened version of Anderson--Wilson's theorem.  
We use this to establish a variant of Anderson--Wilson's  theorem for the Menger curve $|\mathbb{M}|$; see Corollary \ref{C:MengerProjUniversality}. 
Notice that the following weak version of Corollary \ref{C:MengerProjUniversality} already follows from the projective extension property of $\mathbb{M}$ and 
Theorem \ref{T: Peano <--> prespaces}.

\begin{proposition}\label{P: universality}
Every Peano curve $X$ is the continuous surjective image of the Menger curve $|\mathbb{M}|$ under a continuous and connected map $|h|\colon |\mathbb{M}|\to X$.
\end{proposition}
\begin{proof}
By Theorem \ref{T: Peano <--> prespaces}, the space $X$ is homeomorphic to $|K|$ for some prespace $K=\varprojlim (K_n,g^{n}_m)\in\mathcal{C}^{\omega}$. By the first property of Theorem \ref{T: characterization} we get a connected epimorphism $h_0\colon \mathbb{M}\to K_0$. We lift $h_0$ to a connected epimorphism $h\colon \mathbb{M}\to K$ by repeated application of the second property of Theorem \ref{T: characterization}. Since $h$ is a graph homomorphism cliques in $\mathbb{M}$ map to cliques in $K$. As a consequence $h$ induces a map $|h|\colon \mathbb{M}\to|K|$ between the quotients which is easy to see that it is continuous and connected. 
\end{proof}

To strengthen the features of the map $h$ in Proposition \ref{P: universality} we will isolate certain combinatorial properties of $\mathcal{C}$ and incorporate them in 
the construction of the map $h$ above. Our arguments can be adapted to other \Fraisse{} classes $\mathcal{F}$ which satisfy the analogous properties.

\begin{definition}\label{Def 1-exact}
Let $\mathcal{F}$ be a projective \Fraisse{} class. The projective amalgam $f',g'$ of $f,g$ below is called {\bf structurally exact} (with respect to $f$), if for every $B_0\subseteq B$  with
  $f\res B_0 $ in  $\mathcal{F}$, if we set $D_0= (f')^{-1}(B_0)$, then we have $D_0\in \mathcal{F}$ and $g'\res D_0 \in\mathcal{F}$.
\begin{center}
\medskip
\begin{tikzcd}
D \arrow[d, "g'", swap]\arrow[r, "f'"]  &  B \arrow[d,"f"]\\
C \arrow[r, "g"] & A
\end{tikzcd}
\medskip
\end{center}
 We say that $\mathcal{F}$ has {\bf structurally exact amalgamation} if every $f,g$ as above admit structurally exact amalgam. We say that $\mathcal{F}$ has {\bf two--sided structurally exact amalgamation} if every $f,g$ as above admit an amalgam that is structurally exact with respect to both $f$ and $g$.
\end{definition}

Structural exactness is a natural generalization of the well studied notion of \emph{exactness}. Recall that an amalgamation diagram, as in Definition \ref{Def 1-exact}, is exact if for every $b\in B, c\in C$ with $f(b)=g(c)$ there is $d\in D$ so that $f'(d)=b$ and $g'(d)=c$; see \cite{Ge}. 
In the context of Proposition \ref{P: universality}, structural exactness of $\mathcal{C}$ will allow us to strengthen the connectedness properties of the map $h$. Two--sided structural exactness together with the next property will additionally allow us to control isomorphism type of the fibers of $h$.

\begin{definition}
Let $\mathcal{F}$ be a projective \Fraisse{} class. We say that $\mathcal{F}$ admits {\bf local refinements} if for every $f\colon B_0\to A_0$ in $\mathcal{F}$ and every embedding $i\colon A_0\to A$, there is  $g\colon B\to A$ in $\mathcal{F}$ and an embedding  $j\colon B_0\to B$ so that $g\circ j= i\circ f$.
\end{definition}

\begin{lemma}
The class $\mathcal{C}$ has two--sided structurally exact amalgams and local refinements.
\end{lemma}
\begin{proof}
The amalgam provided in the proof of Theorem \ref{T:1} is structurally exact with respect to both $f$ and $g$ as well. It is also easy to check that  $\mathcal{C}$ admits local refinements.
\end{proof}

We can now prove the main theorem of this section.

\begin{theorem}\label{T:preMengerProjUniversality}
For every $K\in\mathcal{C}^{\omega}$ there exists a connected epimorphism $h\colon \mathbb{M}\to K$ which is open and satisfies the following properties:
\begin{enumerate}
\item for every $x\in \mathbb{M}$ there exists a collection $\mathcal{N}$ of clopen subsets of $\mathbb{M}$, with $\bigcap\mathcal{N}=[x]$, so that for every $N\in\mathcal{N}$ and for every closed connected subgraph $F$ of $h(N)\subset K$ the subgraph $h^{-1}(F)\cap N$ of $\mathbb{M}$ is connected;
\item for every closed subgraph $Q$ of $K$ that is a clique, the subgraph $h^{-1}(Q)$ of $\mathbb{M}$ is isomorphic to $\mathbb{M}$.
\end{enumerate}
\end{theorem}
\begin{proof}
Fix  sequences $(M_n,f^{n}_m)$ and $(K_n,g^{n}_m)$ in $\mathcal{C}$ with $\mathbb{M}=\varprojlim(M_n,f^{n}_m)$ and $K=\varprojlim(K_n,g^{n}_m)$. We denote by $f_n$ and $g_n$ the induced maps $\mathbb{M}\mapsto M_n$ and $K \mapsto K_n$.

 We will first use the fact that $\mathcal{C}$ has structurally exact amalgams to produce map $h\colon \mathbb{M}\to K$ in $\mathcal{C}^{\omega}$ which is open and satisfies the property (1) in the statement of the theorem. Then, we will illustrate how to adjust the construction to additionally fulfill property (2) of the statement. 
 We point out that part of the argument below---deriving from exactness that the map $h$ is open---can also be found in \cite{Ge}.

We build $h$ as an inverse limit of a coherent sequence of maps $h_i\colon M_{n(i)}\to K_i$ from $\mathcal{C}$ where $(n(i): i\in\mathbb{N})$  is some increasing sequence of natural numbers. By the first property of Theorem \ref{T: characterization} we get $n(0)$ and a connected epimorphism $h_0\colon M_{n(0)}\to K_0$. Assume now that we have defined $n(i)$ and $h_i$. Setting $f=h_i$ and $g=g^{i+1}_{i}$ in initial diagram of Definition \ref{Def 1-exact} we get a structurally exact amalgam $D$, $f'\colon D\to M_{n(i)}$, $g'\colon D \to K_{i+1}$. Using the extension property of 
Theorem \ref{T: characterization} we find $n(i+1)$ and a map $p\colon M_{n(i+1)}\to D$ such that $f' \circ p=f^{n(i+1)}_{n(i)}$. Set $h_{i+1}=p\circ g'$.  
This finishes the induction and we therefore get a map $h=\varprojlim(h_i)$ in $\mathcal{C}^{\omega}$ from $\mathbb{M}$ to $K$.

\begin{claim}
For every $i\in\mathbb{N}$ and $a\in M_{n(i)}$ we have that $h(f_{n(i)}^{-1}(a))=g^{-1}_i(h_i(a))$.
\end{claim}
\begin{proof}[Proof of Claim.]
The non-trivial direction, $h(f_{n(i)}^{-1}(a)) \supseteq g^{-1}_i(h_i(a))$, follows from exactness of $D$ in the inductive step above. In particular, let $x=(x_0,x_1,\ldots)\in K$ with $x_i=h_i(a)$ and let $y_i=a$. Then, since $D$ is exact, there is $d\in D$ with $f'(d)=y_i$ and $g'(d)=x_{i+1}$. Let $y_{i+1}=d'$ for any $d'\in p^{-1}(d)$. Continuing this way we build inductively $y=(y_0,y_1,\ldots)\in \mathbb{M}$ with $h(y)=x$.
\end{proof}

By the above claim, the fact that $h_i$ is open, and since the family of all sets of the form $f_{n(i)}^{-1}(a)$ forms a basis for the topology of $\mathbb{M}$, it follows that $h$ is open.

Next we show that $h$ satisfies Property (1) in the statement of the Theorem.  Let $x\in\mathbb{M}$ and notice that for every $n$, the subgraph $Q_{x,n}=f_n([x])$ of $M_n$ is a clique (of size at most $2$). 
Set $\mathcal{N}=\{ f_n^{-1}(Q_{n}) : n\in\mathbb{N}\}$ and notice that by Lemma \ref{L: basis} it follows that $\mathcal{N}$ is indeed a collection of clopen subsets of $\mathbb{M}$ with $\bigcap\mathcal{N}=[x]$.  Let $N\in\mathcal{N}$ and let $F$ be a closed connected subgraph of $h(N)\subset K$. By reparametrizing the sequences $(M_n)$ and $(K_n)$ above we can assume that  $N=f^{-1}_0(Q)$ for some clique $Q$ in $M_0$ and that  $n=n(i)=i$ in the definition of the sequence $h_i$ above.  Set $Q_n=(f^n_0)^{-1}(Q)$  and let $F_n=g_n(F)$. It is immediate that $F_n$ is a connected subgraph of $K_n$ included in $h_n(Q_n)$, for every $n\in\mathbb{N}$. Let $E_n=h^{-1}_n(F_n)\cap Q_n$. While $f^n_m\res E_n$ could fail to be a connected epimorphism, the following claim is true: 

\begin{claim}
$E_n$ is a connected subgraph of $Q_n$.
\end{claim}
\begin{proof}
We prove this inductively. To run the induction we will actually need the stronger statement that $h_n\res E_n\colon E_n\to F_n$ is in $\mathcal{C}$. Let $E_0=h_0^{-1}(F_0)\cap Q_0$. Since $Q_0$ is a clique, $h_0\res E_0$ is a connected epimorphism from $E_0$ onto $F_0$.

Assume now that $h_n\res E_n\colon E_n\to F_n$ is in $\mathcal{C}$. Since the structural exactness of $D,f',g'$ at the stage $n$ of the construction above is stable under precomposing with $p\colon M_{n+1}\to D$, we have that $h_{n+1}\res (f^{n+1}_n)^{-1}(E_n)$ is a connected epimorphism from $(f^{n+1}_n)^{-1}(E_n)$ to $(g^{n+1}_n)^{-1}(F_n)$. Notice now that $E_{n+1}$, that was defined as $h_{n+1}^{-1}(F_{n+1})\cap Q_{n+1}$, equals $h_{n+1}^{-1}(F_{n+1})\cap (f^{n+1}_n)^{-1}(E_n)$. Since $E_{n+1}$ is the preimage of the connected set $F_{n+1}$ under the connected epimorphism $h_{n+1}\res (f^{n+1}_n)^{-1}(E_n)$, the map $h_{n+1}\res E_{n+1}\colon E_{n+1}\to F_{n+1}$ is connected as well.
\end{proof}
Since $f^n_m(E_n)=E_m$, the above claim implies that the inverse limit 
\[
E=\varprojlim (E_n,f^{n}_m\res E_n)
\] 
is a closed and connected subgraph of $\mathbb{M}$ (although not in general locally-connected), with 
$h^{-1}(F)\cap N= E$. Hence indeed $h$ satisfies the  property (1)  above.

We finish by describing how the above construction can be modified so that $h$ additionally satisfies property $(2)$ of the statement. Recall from Section \ref{S: Menger curve} that any topological graph is isomorphic to $\mathbb{M}$ if it can be expressed as an inverse limit of a generic sequence $(L_n,t^n_m)$. Recall also that a sequence $(L_n,t^n_m)$ is generic if is ``saturated'' with respect to $(A_n)$ and $(e_n)$; see the construction in Section \ref{S: Menger curve}.

Let $\mathcal{Q}$ be the collection of all closed subsgraphs of $K$ which are cliques.
Fix $Q\in\mathcal{Q}$ and for each $i$ set $Q_i=g_i(Q)$ and $L^{Q}_i=h_i^{-1}(Q_i)$. Notice that $Q_i$ is a clique in $K_i$ and as a consequence $g^{i+1}_i\res Q_{i+1}$ is a connected epimorphism from $Q_{i+1}$ to $Q_i$.  
Hence, by assuming during the construction of $h_{i+1}$ in the above that the amalgam $f'\colon D\to M_{n(i)}$,   $g'\colon D \to K_{i+1}$  is two--sided structurally exact, we have that $f'\res (f')^{-1}(Q_{i+1})$ is in $\mathcal{C}$, and therefore, $f^{n(i+1)}_{n(i)}\res L^Q_{i+1} \colon L^Q_{i+1}\to L^Q_i$ is in $\mathcal{C}$.  Therefore, for every $Q\in\mathcal{Q}$ we already have that $h^{-1}(Q)=\varprojlim(L^{Q}_n,f^{n}_m\res L^{Q}_n)\in\mathcal{C}^{\omega}$.  

In order to arrange for $h$ to have property (2) we need to make sure that for every $Q\in\mathcal{Q}$ the sequence $(L^{Q}_n,f^{n}_m\res L^{Q}_n)$ is generic. This is done by modifying slightly the definition of  $h_i$ above. In particular, let $(A_n)$ and $(e_n\colon C_n\to B_n)$ be as in the construction described in Section \ref{S: Menger curve} and assume that for every $Q\in\mathcal{Q}$ the finite sequence $(L^{Q}_n,f^{n}_m\res L^{Q}_n; n,m\leq i)$ has been saturated with respect to $(A_n; n \leq i)$ and $(e_n; n \leq i)$.
 In the process of defining $h_{i+1}$, after we construct $D,f',g'$ as the two--sided structurally exact amalgam of $h_i$ and $g^{i+1}_i$, we further refine it via a map $r\colon D'\to D$ in $\mathcal{C}$ which makes sure that if
$r'=f'\circ r$ is the map from  $D'$ to  $M_{n(i)}$ then for every $Q\in \mathcal{Q}$ we have that:
\begin{enumerate}
\item[(i)] there exists a map in $\mathcal{C}$ from  $(r')^{-1}(Q_{i+1})$ to $A_{i+1}$;
\item[(ii)] if $s\in\mathcal{C}$ is any map from $(f')^{-1}(Q_{i+1})$ to $B_{i+1}$ then there exists $d\in\mathcal{C}$ from $(r')^{-1}(Q_{i+1})$ to $C_{i+1}$  so that          $ s \circ \big(r\res (r')^{-1}(Q_{i+1})\big) = e_{i+1} \circ d$.
\end{enumerate}
This is easily done since the ``local problems''  (i) and (ii) can be turned into  ``global problems'' given  that $\mathcal{C}$ has the local refinement property, and then get solved using finitely many application of the amalgamation property of $\mathcal{C}$.

Going back to the construction of $h_{i+1}$ above, we can now use the extension property of $\mathbb{M}$ to find $n(i+1)$ and a map $p\colon M_{n(i+1)}\to D'$ such that $f' \circ r \circ p=f^{n(i+1)}_{n(i)}$ and set $h_{i+1}=p \circ r \circ g'$.
\end{proof}

As a corollary we get the following variant of Anderson--Wilson's projective universality theorem \cite{An2,Wi}. Notice that the corresponding map in \cite{An2,Wi} is shown to be \emph{monotone}, that is, preimages of points are connected. Since we are working with compact spaces, a map is monotone if and and only if 
it is connected \cite[p.131]{Ku}.  
Moreover, as pointed out by Gianluca Basso, the map we construct is not open. Instead we get that  
it is \emph{weakly locally-connected}: a continuous $\phi\colon Y\to X$ between topological spaces is called {\bf weakly locally-connected} if $Y$ admits a collection $\mathcal{N}$ of neighborhoods so that $\{\mathrm{int}(N)\colon N\in \mathcal{N}\}$ generates the topology of $Y$ and for every $N\in\mathcal{N}$, and for every closed subset $Z$ of $\phi(\mathrm{int}(N))$ we have that $\phi^{-1}(Z)\cap N$ is connected. This property seems rather technical but is very useful for constructing nice sections for the map $\phi$; see \cite{Mi}.

\begin{corollary}[see also Anderson \cite{An2}, Wilson \cite{Wi}] \label{C:MengerProjUniversality}
If  $X$ is a Peano continuum, then there exists a continuous surjective map $|h|\colon |\mathbb{M}|\to X$ which is connected, weakly locally-connected, and  $|h|^{-1}(x)$ is homeomorphic to $|\mathbb{M}|$, for every  $x\in X$.
\end{corollary}
\begin{proof}
By Theorem \ref{T: Peano <--> prespaces}, the space $X$ is homeomorphic to $|K|$ for some prespace $K=\varprojlim(K_n,g^{n}_m)$ in $\mathcal{C}^{\omega}$. Let $h\colon \mathbb{M}\to K$ be the map provided by Theorem \ref{T:preMengerProjUniversality}. Since $h$ is an $R$-homomorphism, the map $h$ induces a map $|h|\colon |\mathbb{M}|\to |K|$ between the quotients. It is easy to check that $|h|$ continuous and surjective and connected. The rest follow from properties (1) and (2) of Theorem \ref{T: Peano <--> prespaces}.
\end{proof}

\section{The approximate projective homogeneity property}\label{S:ApproximateProjectiveHomogeneity}

The Menger prespace $\mathbb{M}$, being the projective \Fraisse{} limit of $\mathcal{C}$, 
automatically enjoys the \emph{projective homogeneity property}: for every $f,g\colon\mathbb{M}\to A$, with $A\in\mathcal{C}$ and $f,g\in\mathcal{C}^{\omega}$, 
there is $\phi\in\mathrm{Aut}(\mathbb{M})$ with $f\circ\phi=g$. From this property we can naturally derive the following  {\bf approximate projective homogeneity property} for the Menger curve $|\mathbb{M}|$.

\begin{theorem}\label{T:ApproximateHomogeneityProperty}
If  $\gamma_0,\gamma_1\colon |\mathbb{M}|\to X$ are continuous and connected maps from the Menger curve onto some Peano continuum $X$, then for every open cover $\mathcal{V}$ of $X$ there is 
 $h\in \mathrm{Homeo}(|\mathbb{M}|)$ so that  $(\gamma_0\circ h)$ and $\gamma_1$ are $\mathcal{V}$-close, that is, 
 \[\forall y\in |\mathbb{M}| \; \exists V\in\mathcal{V} \; (\gamma_0\circ h)(y),\gamma_1(y)\in V.\]
\end{theorem}

In other words, if we endow  the space $\mathrm{Maps}_0(|\mathbb{M}|,X)$, of all continuous and connected maps from $|\mathbb{M}|$ onto the Peano continuum $X$ with the compact open topology, then the orbit of each $\gamma\in \mathrm{Maps}_0(|\mathbb{M}|,X)$ under the natural action of $ \mathrm{Homeo}(|\mathbb{M}|)$ on $\mathrm{Maps}_0(|\mathbb{M}|,X)$ is dense in $\mathrm{Maps}_0(|\mathbb{M}|,X)$. We start with a lemma.

\begin{lemma}\label{L:ApproximateHomogeneityProperty}
Let $A\in\mathcal{C}$ and let $\mathcal{U}=\{U_a\colon a\in \mathrm{dom}(A)\}$ be an open cover of $\mathbb{M}$ consisting of connected subgraphs. If $U_a\cap U_b\neq\emptyset\iff R^{A}(a,b)$, then there is $u\colon \mathbb{M}\to A$ in $\mathcal{C}^{\omega}$ so that $u^{-1}(a)\subseteq U_a$, for all $a\in\mathrm{dom}(A)$.
\end{lemma}
\begin{proof}
First we pick for each each $a\in\mathrm{dom}(A)$  a clopen connected subgraph $W_a$ of $\mathbb{M}$, with $\mathrm{dom}(W_a)\subseteq \mathrm{dom}(V_a)$, so that 
\begin{equation}\label{Equatio2}
W_a\cap W_b\neq \emptyset \text{ if and only if } R^A(a,b).
\end{equation}
This can always be arranged  as follows. Let $f_0\colon \mathbb{M}\to B$ be any map in $\mathcal{C}^{\omega}$, with $B\in\mathcal{C}$, so that $\{f_0^{-1}(b): b\in\mathrm{dom}(B)\}$ refines $\mathcal{U}$.  Let $B\times Q_A\in\mathcal{C}$ be the product---see proof of Theorem \ref{T:1}---of $B$ with the clique $Q_A$ on domain $\mathrm{dom}(A)$, and let $p\colon B\times Q_A\to B$ the natural projection. Let  $C\in\mathcal{C}$ be the graph attained by subdividing every non-trivial edge of the graph $B\times Q_A$, and let $r\colon C\to B\times Q_A$ be any map which maps every vertex that came from a subdivision to either of its two neighbors; and every vertex already in $\mathrm{dom}(B\times Q_A)$  to itself. Clearly the map  $s\colon C\to B$ with   $s = p \circ r$ is in $\mathcal{C}$. By the projective extension property of $\mathbb{M}$---see; Theorem \ref{T: characterization}---we can replace $f_0$ with a map $f\colon \mathbb{M}\to C$  from $\mathcal{C}^{\omega}$. Notice that for the map $f$  we can choose: for every $a\in\mathrm{dom}(A)$, a vertex $v_a\in \mathrm{dom}(C)$   with $f^{-1}(v_a)\subseteq U_a$, so that $v_a\neq v_{b}$ if $a\neq b$; and for every $a,b\in \mathrm{dom}(A)$ with $R^A(a,b)$, a path $P:=P(a,b)$ in $C$ from $v_a$ to $v_b$, with $f^{-1}(P)\subseteq U_a\bigcup U_{b}$, so that the collections of all these paths  forms a ``strongly pairwise disjoint'' system,  i.e.,  if the paths $P,P'$ are distinct and $v\in P$, $v'\in P'$, with $R^C(v,v')$, then either $v$ is an endpoint of $P$ or $v'$ is an endpoint of $P'$. Using this ``strongly pairwise disjoint" system of paths  it is easy to define the collection $\{W_a: a\in\mathrm{dom}(A)\}$.
  
 Next we find  clopen, connected subgraphs $\widetilde{W}_a$ of $\mathbb{M}$ with
\begin{equation}\label{Equatio3}
\mathrm{dom}(W_a)\subseteq \mathrm{dom}(\widetilde{W}_a)\subseteq U_a, \; \widetilde{W}_a\cap\widetilde{W}_{a}'=\emptyset\text{ if }  a\neq a', \; \mathrm{dom}(\mathbb{M})=\bigcup_{a}\mathrm{dom}(\widetilde{W}_a),
\end{equation}
and   define the map $u\colon \mathbb{M}\to A$ with  $u^{-1}(a)=\widetilde{W}_a$. Properties (\ref{Equatio2}), (\ref{Equatio3}), and the fact that $U_a\cap U_b\neq\emptyset\iff R^{A}(a,b)$ will then  imply that this is indeed the desired map.  

We define $\widetilde{W}_a$ as the union $\bigcup_{n}W^n_a$ of an increasing sequence of clopen subgraphs of  $U_a$. Let $(O_k)$ be an enumeration of a basis for the topology of $\mathrm{dom}(\mathbb{M})$ consisting of clopen connected graphs with the property that $O_k\cap U_a\neq\emptyset$ implies  $O_k\subseteq U_a$ for all $a\in\mathrm{dom}(A)$ and $k\in\mathbb{N}$.  We set $W^0_a=W_a$, for every $a\in\mathrm{dom}(A)$.  Assume that $W^n_a$ has been defined for all $a\in\mathrm{dom}(A)$, and let $k(n+1)$ be the smallest natural number so that $O_{k(n+1)}\cup (\bigcup_a W^n_a)$ is a connected graph strictly expanding $\bigcup_a W^n_a$, if such $k$ number exists; otherwise, let $k(n+1)=\infty$. If $k(n+1)\in\mathbb{N}$ then  $O_{k(n+1)}$ is compact and locally-connected. Hence, $O_{k(n+1)}\setminus (\bigcup_a W^n_a)$ is the union of finitely many clopen connected subgraphs $R_1,\ldots,R_m$ of $\mathbb{M}$. It is easy to see that for each $i\leq m$ there is some $a(i)$ so that $W^n_{a(i)}\cup R_i$ is connected. Let $W^{n+1}_a$ be the union of $W^n_a$ together with all $R_i$ with $a(i)=a$, if $k(n+1)\in\mathbb{N}$; and let $W^{n+1}_a=W^n_a$, if $k(n+1)=\infty$. This finishes the definition of $\{W^n_a: a\in \mathrm{dom}(A)\}$ for each $n\in\mathbb{N}$ and an easy induction shows that  $\{W^n_a: a\in \mathrm{dom}(A)\}$ is a disjoint collection of clopen connected graphs with $W^n_a\subseteq U_a$. We are left to show that 
\[\mathbb{M}=\bigcup_n \bigcup_{a} W^n_a,\]
since then, by compactness of $\mathrm{dom}(\mathbb{M})$, the  union along $\mathbb{N}$ 
will stabilize at some finite $n$, and  $\widetilde{W}_a=\bigcup_nW^n_a$ will therefore be clopen. Assume towards contradiction that some $x\in\mathrm{dom}(\mathbb{M})$ is not in the domain of the above union and let $k(x)$ be such that $x\in O_{k(x)}$. It follows that  
\begin{equation}\label{Equatio4}
[O_{k(x)}]\cap \bigcup_n \bigcup_{a} W^n_a =\emptyset,
\end{equation}
since otherwise $[V_{k(x)}]\cap \bigcup_{n\leq l} \bigcup_{a} W^n_a\not=\emptyset$ for some $l$, implying that  for each $n>l$, $k(n)\not= k(n+1)$ and 
$k(n)<k(x)$, which is contradictory. But then, setting $X=\mathrm{dom}(\mathbb{M})\setminus \bigcup_n \bigcup_{a} \mathrm{dom}(W^n_a)$, we have by (\ref{Equatio4}) that:
\[
\mathbb{M}=  [\bigcup_{x\in  X}O_{k(x)}]   \bigcup  \big( \bigcup_n \bigcup_{a} W^n_a \big), \text{ with } [\bigcup_{x\in  X}O_{k(x)}]   \bigcap  \big( \bigcup_n \bigcup_{a} W^n_a \big) =\emptyset, 
\]
Contradicting that $\mathbb{M}$ is a connected graph.
\end{proof}

We can now finish the proof of Theorem \ref{T:ApproximateHomogeneityProperty}.
\begin{proof}[Proof  of Theorem \ref{T:ApproximateHomogeneityProperty}]

By Theorem \ref{T: Peano <--> prespaces}, $X$ is homeomorphic to $|K|=\pi_K(K)$ for some prespace $K\in\mathcal{C}^{\omega}$. Let $g\colon K\to A$ be a map in $\mathcal{C}^{\omega}$ with $A\in\mathcal{C}$ so that $\{\pi_K(g^{-1}(a))\mid a\in \mathrm{dom}(A)\}$ refines $\mathcal{V}$. Since each $\pi_K(g^{-1}(a))$ is a compact and connected subset of a locally-connected space we can find connected open subsets $V_a\supseteq \pi(g^{-1}(a))$ of $|K|$, with $V_a\cap V_b\neq \emptyset$ if and only if $R^A(a,b)$,  so that $\{V_a\mid a\in\mathrm{dom}(A)\}$ refines $\mathcal{V}$. Let $U^0_a:=   (\gamma_0\circ \pi_{\mathbb{M}})^{-1}(V_a), U^1_a:=   (\gamma_1\circ \pi_{\mathbb{M}})^{-1}(V_a)$, and set $\mathcal{U}^0:=\{U^0_a\mid a\in\mathrm{dom}(A)\}, \mathcal{U}^1:=\{U^1_a\mid a\in\mathrm{dom}(A)\}$. Then $\mathcal{U}^0$ and $\mathcal{U}^1$ are open covers of $\mathbb{M}$ consisting of  connected graphs of $\mathbb{M}$  so that: 
\begin{equation}\label{Equatio1}
U^0_a\cap U^0_b\neq  \emptyset \iff R^A(a,b)  \iff   U^1_a\cap U^1_b\neq \emptyset
\end{equation}
To see  that $U^0_a$ and $U^1_a$ are connected graphs, notice that, since $X$ is a Peano continuum,  $V_a$ is the increasing union of compact connected sets, and since $\gamma$ is a connected map,  $\gamma^{-1}(V_a)$  is also the increasing union of compact connected sets.

Let  $u_0$ and  $u_1$ be the maps  given by applying Lemma \ref{L:ApproximateHomogeneityProperty} to the covers $\mathcal{U}_0$ and $\mathcal{U}_1$, respectively. By the projective homogeneity property of $\mathbb{M}$ there is $\varphi\in\mathrm{Aut}(\mathbb{M})$ so that  $u_0\circ\varphi=u_1$. Let $h\colon |\mathbb{M}|\to |\mathbb{M}|$ with $h([x])=(\pi_{\mathbb{M}}\circ \varphi)(x)$. Since $\varphi\in\mathrm{Aut}(\mathbb{M})$, it follows that $h$  is a well-defined homeomorphism of $|\mathbb{M}|$. To check that this is the desired homeomorphism, let $y\in|\mathbb{M}|$  and fix $x\in \mathrm{dom}(\mathbb{M})$ with $\pi_{\mathbb{M}}(x)=y$. Set $a:=u_1(x)$ and notice that since $x\in u_1^{-1}(a)\subseteq (\gamma_1\circ \pi_{\mathbb{M}})^{-1}(V_a)$, we have that $\gamma_1(y)=\gamma_1 \circ \pi_{\mathbb{M}}(x)\in V_a$. On the other hand, $h(y)=\pi_{\mathbb{M}}(\varphi(x))$, and since $u_0^{-1}(a)\subseteq (\gamma_0\circ \pi_{\mathbb{M}})^{-1}(V_a)$, we have that $\varphi(x)\in u^{-1}_0(a)\subseteq (\gamma_0\circ \pi)^{-1}(V_a)$. Hence, $h(y)=\gamma_0\circ \pi_{\mathbb{M}}\circ \varphi(x)\in V_a$.
 Since $\{V_a:\mathrm{a}\in\mathrm{dom}(A)\}$ refines  $\mathcal{V}$, we are done.
\end{proof}

\section{The $n$-dimensional case}\label{S: n-dim}

In this section, we consider simplicial complexes that are more general than graphs. A {\bf simplicial complex}  $C$ is a family of finite sets  that is closed downwards, that is, if $\sigma\in C$ and $\tau\subset\sigma$  then $\tau\in C$. The elements $\sigma$ of $C$ are called {\bf faces} of the simplicial complex. 
We set $\mathrm{dom}(C)=\cup C$ to be the {\bf domain} of the simplicial complex. A {\bf subcomplex} $D$ of $C$ is a simplicial complex with $D\subseteq C$.  A {\bf simplicial map} $f\colon B\to A$ is a map from $\mathrm{dom}(B)$ to $\mathrm{dom}(A)$ with $f\sigma\in A$ whenever $\sigma\in B$, where $f\sigma$ stands for the set $\{f(v) \colon v\in\sigma\}$.

Let $C$ be simplicial complex and let $\sigma\in C$. The {\bf dimension} $\mathrm{dim}(\sigma)$ of $\sigma$ is $n\geq(-1)$ if the cardinality of $\sigma$ is $n+1$.  We say that $C$ is {\bf $n$-dimensional} if $\mathrm{dim}(\sigma)\leq n$ for every $\sigma\in C$. We briefly recall some definitions from algebraic topology. For more details see Definition \ref{D:Acyclic} and the discussion after the proof of Theorem \ref{T:n-dim}.
 We say that $C$ is {\bf $n$-connected} if all homotopy groups $\pi_k(C)$ of $C$, with $k\leq n$, vanish. We say that it is {\bf $n$-acyclic} if all (reduced) homology groups $\widetilde{H}_k(C)$ of $C$, with $k\leq n$, vanish.
 Similarly, a simplicial map $f\colon B\to A$ is called {\bf $n$-connected} if the preimage of every $n$-connected subcomplex of $A$ under $f$ is $n$-connected, and it is called {\bf $n$-acyclic} if the preimage of every $n$-acyclic subcomplex of $A$ under $f$ is $n$-acyclic. Since a simplicial complex $A$ is $(-1)$-connected if and only $\mathrm{dom}(A)\neq \emptyset$, a simplicial map is $(-1)$-connected if it is a surjection on the domains of the simplicial complexes.

\begin{definition}
For every $n\in\{0,1,\ldots\}\cup\{\infty\}$, let $\mathcal{C}_n$ be the class of all $(n-1)$-connected simplicial maps between finite, $n$-dimensional, $(n-1)$-connected simplicial complexes.
Similarly let $\widetilde{\mathcal{C}}_n$ be the class of all $(n-1)$-acyclic simplicial maps between finite, $n$-dimensional, $(n-1)$-acyclic simplicial complexes.
\end{definition}

\begin{theorem}\label{T:n-dim}
For all $n$ as above, both $\mathcal{C}_n$ and $\widetilde{\mathcal{C}}_n$ are projective \Fraisse{}.
\end{theorem}
For the proof of Theorem \ref{T:n-dim} will need the next lemma. Let $\rho$ be a finite set.  The {\bf simplex $\Delta(\rho)$ on $\rho$} is the simplicial complex $\{\sigma \colon \sigma\subseteq\rho\}$. If $C$ is a simplicial complex and $\rho\in C$ then $\Delta(\rho)$ is a  subcomplex  of $C$. 

\begin{lemma}\label{L:n-dim}
If  $f\colon B\to A$ is a simplicial map between two finite simplicial complexes, then we have that:
\begin{enumerate}
\item $f$ is  $(n-1)$-connected if and only if   $f^{-1}(\Delta(\rho))$ is $(n-1)$-connected for every $\rho\in A$ with $\mathrm{dim}(\rho)\leq n$.
\item $f$ is  $(n-1)$-acyclic if and only if   $f^{-1}(\Delta(\rho))$ is $(n-1)$-acyclic for every $\rho\in A$ with $\mathrm{dim}(\rho)\leq n$.
\end{enumerate}
\end{lemma}

Before we discuss the proof of Lemma  \ref{L:n-dim} we show how it implies Theorem \ref{T:n-dim}.

\begin{proof}[Proof of Theorem \ref{T:n-dim}.]
We just check here the projective amalgamation property. Fix $n$ and let $f\colon B\to A$ and $g\colon C\to A$ be maps in $\mathcal{C}_n$. We will	 define the projective amalgam $D,f',g'$ as the $n$--skeleton $\mathrm{Sk}^n(B\times_A C)$ of the simplicial pullback $B\times_A C$, together with the canonical projection maps $\pi_B,\pi_C$. Recall that the simplicial pullback $B\times_A C$ is defined on domain $\mathrm{dom}(B)\times_{\mathrm{dom}(A)}\mathrm{dom}(C)$ as the simplicial complex whose faces are precicely all sets of the form
\[\sigma\times_A\tau=\{(b,c)\colon b\in\sigma, \; c\in \tau, \; f \sigma=g \tau \}, \]
where $\sigma\in B$ and $\tau\in C$. We let $D$ be the simplicial complex  attained by $B\times_A C$ after we omit all faces of dimension strictly larger than $n$. Let  $f'=\pi_B$ and  $g'= \pi_C$ be the projection maps $(b,c)\mapsto b$ and $(b,c)\mapsto c$ from $D$ to $B$ and $C$ respectively. It is easy to check that both $f',g'$ are simplicial epimorphisms (surjective on faces). We now check that $f'\colon D\to B$ is $(n-1)$-connected.  The fact that $D$ is $(n-1)$-connected is a special case of this and the fact. The same argument applies symmetrically to $g'$. 

Let $B_0$ be a $(n-1)$-connected subcomplex of $B$. To show that $D_0=(f')^{-1}(B_0)$ is $(n-1)$-connected it suffices by Lemma \ref{L:n-dim}(1) to show that $(f')^{-1}(\Delta(\rho))$ is $(n-1)$-connected for every $\rho\in B$. Let $\tau$ be the image of $\rho$ under $f$ and let $\Delta(\tau)$ the corresponding simplex, that is a subcomplex of $A$. Let also $C_0=g^{-1}(\Delta_{\tau})$ and notice that, since $g\in\mathcal{C}_n$, $C_0$ is a $(n-1)$-connected subcomplex of $C$. Notice that $C_0$ is isomorphic to the subcomplex $K$ of $\mathrm{Sk}^n(\Delta(\tau) \times_{\Delta(\tau)} C_0)$ spanned by the vertexes in $\mathrm{graph}^{*}(g \res \mathrm{dom} (C_0) ):=\{(w,v)\in \tau \times\mathrm{dom}(C_0)  \colon g(v)=w\}$, where $\Delta(\tau)\times_{\Delta(\tau)} C_0$ is formed with respect to $\mathrm{id}\colon \Delta(\tau)\to \Delta(\tau)$ and $g\res \mathrm{dom} (C_0) \colon C_0\to \Delta(\tau)$. Now, again by Lemma \ref{L:n-dim}(1), it is easy to see that  the function $(f\res \rho) \times \mathrm{id}$ from $\mathrm{Sk}^n(\Delta(\rho) \times_{\Delta(\tau)} C_0)$ to $\mathrm{Sk}^n(\Delta(\tau) \times_{\Delta(\tau)} C_0)$ is $(n-1)$-connected. But $D_0$ is simply the preimage of $K$ under this map and $K$ is isomorphic to $C_0$ which is $(n-1)$-connected.
A similar argument, using Lemma \ref{L:n-dim}(2) instead of Lemma \ref{L:n-dim}(1), shows that $\widetilde{\mathcal{C}}_n$ satisfies the projective amalgamation property.
\end{proof}

Lemma \ref{L:n-dim} (1) and (2) are special cases of \cite[Proposition 7.6]{Qu} and \cite[Corollary 4.3]{Bj}, respectively. However, since we are dealing with finite combinatorial objects, one can provide a direct proof of Lemma \ref{L:n-dim}. In the rest of this section  we sketch the steps for a hands-on proof Lemma \ref{L:n-dim} (2). The interested reader can fill the missing details. For Lemma \ref{L:n-dim} (1) recall that, by the Hurewicz Theorem, a simplicial complex is $(n-1)$-connected  for $n\geq 2$, if it is $(n-1)$-acyclic and it has a trivial fundamental group.   A combinatorial proof of Lemma \ref{L:n-dim} (1) is now possible using the notions of \emph{combinatorial paths} and \emph{combinatorial homotopy} from  \cite{HW}. 

We now recall from \cite{Po} basic notions from homology  and the proceed to sketch a direct proof of Lemma \ref{L:n-dim} (2). Let $C$ be a simplicial complex and let $\sigma\in C$. An {\bf orientation} for $\sigma$ is an equivalence class  of expressions $\epsilon(v_0,\ldots,v_n)$, where $\sigma=\{v_0,\ldots,v_n\}$ and $\epsilon\in\{-1,1\}$. For $n=-1$ we have the empty listing. Two such expressions $\epsilon(v_0,\ldots,v_n)$ and $\epsilon'(v'_0,\ldots,v'_n)$ are equivalent if for the unique permutation $\pi$ with $v_i=v'_{\pi(i)}$, we have that $\mathrm{sgn}(\pi)=\epsilon\epsilon'$. There are precisely two orientations associated with each face. An {\bf oriented face} $\orient{\sigma}$ in $C$ is just an orientation for $\sigma$ with $\sigma\in C$.

The {\bf chain group} $\mathbb{C}(C)$ of a complex $C$ is the abelian group generated by oriented faces of $C$, with the relations $\orient{\sigma}+\orient{\tau}=0$, for any two distinct oriented faces $\orient{\sigma}$ and $\orient{\tau}$ with $\sigma=\tau$. Elements of $\mathbb{C}(C)$ are called {\bf chains}.
Each chain is uniquely represented as a finite sum $\sum_i \orient{\sigma}_i$, where each $\orient{\sigma}_i$ is 
an oriented face  and, for all $i,j$, if $\sigma_i = \sigma_j$, then $\orient{\sigma}_i=\orient{\sigma}_j$. We say that $\orient{\sigma}_i$ {\bf is in} the chain $\sum_i \orient{\sigma}_i$. The empty sum represents the identity element $0\in\mathbb{C}(C)$.  An {\bf $n$-chain} is a chain consisting entirely of $n$-dimensional oriented faces. A {\bf $(\leq n)$-chain} consists of oriented faces  whose dimension  is less that or equal to $n$.
The chain group is equipped with an endomorphism $\partial$ which is defined on the generators of
$\mathbb{C}(C)$ by the following procedure. If $\orient{\sigma}$ is one of the two $(-1)$-dimensional oriented faces, let $\partial \orient{\sigma} = 0$.
If $\orient{\sigma}$  is the equivalence class of $\epsilon (v_0, \dots, v_n)$ with  $n\geq 0$, let
\begin{equation}\label{E:boun}
\partial \orient{\sigma} = \sum_{i=0}^n \orient{\sigma}_i,
\end{equation}
where $\orient{\sigma}_i$ is the equivalence class of $(-1)^i\epsilon (v_0, \dots, v_{i-1}, v_{i+1}, \dots, v_n)$.  Let finally $f\colon B\to A$ be a simplicial map. This map induces a function
\[
f_{\#}\colon \mathbb{C}(B)\to \mathbb{C}(A)
\]
given by the following rules. Let $\orient{\sigma}$ be an oriented face
in $B$. If $f \sigma$ has dimension strictly smaller than that of $\sigma$, let $f_\#(\orient{\sigma})=0$. If the dimensions of
$f\sigma$ and $\sigma$ are equal and $\orient{\sigma}$ is the equivalence class of $\epsilon (v_0, \dots, v_n)$, define $f_\#(\orient{\sigma})$ to be the equivalence class of
$\epsilon (f(v_0), \dots, f(v_n))$. One checks that
\[
f_\#\circ \partial = \partial\circ f_\#.
\] 

We have now developed all homological prerequisites for the main definition.
\begin{definition} \label{D:Acyclic}
Let $n\geq (-1)$. A complex $C$ will be called {\bf $n$-acyclic} if for each $(\leq n)$-chain $\zeta$ with $\partial \zeta=0$ there is a chain $\eta$ with $\zeta = \partial \eta$.
\end{definition}

The non-trivial direction of Lemma \ref{L:n-dim} (2) reduces to the following more general statement whose proof relies on  Lemma \ref{L:aux1} and  Lemma \ref{L:aux2}

\begin{lemma}\label{L:l-n-acyclicity}
If $f\colon B\to A$ is a simplicial map between finite simplicial complexes and  for some  $l,n\in\mathbb{N}$ we have that:  
\begin{enumerate}
\item $f$ is still simplicial when viewed as a map from $\mathrm{Sk}^n(B)$ to $\mathrm{Sk}^l(A)$;
\item   $f^{-1}(\Delta(\sigma))$ is $(n-1)$-acyclic  for every $\sigma\in \mathrm{Sk}^l(A)$;
\end{enumerate}
then $B$ is $(n-1)$-acyclic if $A$ is $(l-1)$-acyclic. 
\end{lemma}

For any simplex $\Delta(\rho)$ on a set $\rho$ we define the {\bf boundary} $\mathrm{Bd}(\Delta(\rho))$ of $\Delta(\rho)$ to be the simplicial complex $\Delta(\rho)\setminus\{\rho\}$.

\begin{lemma}\label{L:aux1}
Let $f\colon B\to A$ be a simplicial map  such that $f^{-1}(\Delta(\sigma))$ is $n$-acyclic,  for every $\sigma \in A$. Let $\zeta$ be an $(\leq n)$-chain in $B$ such that 
each  $\orient{\sigma}$ in $\zeta$ we have that  $\mathrm{dim}(\sigma)>\mathrm{dim}(f\sigma)$.  If $\partial\zeta=0$, then there is a chain $\eta$ such that $\zeta=\partial\eta$.  
\end{lemma}
\begin{proof}[Sketch of Proof.]
The proof is by induction on $l=\max\{\mathrm{dim}(f\sigma)\colon \orient{\sigma} \;\text{in}\; \zeta \}$. Let $\zeta=\sum_{\rho}\zeta_\rho+\zeta^{-}$, where $\rho$ varies over all $l$-dimensional faces of $A$ for which there is  a $\orient{\sigma}$ in $\zeta$ with $f\sigma=\rho$, and with $\zeta_{\rho}$ collecting all such $\orient{\sigma}$. Since 
\[0=\partial \zeta =\sum_{\rho}\partial \zeta_{\rho}+\partial\zeta^{-}\]
and each $\partial\zeta_{\rho}$ is a chain in $f^{-1}(\Delta(\rho))$ it follows actually that each $\partial\zeta_{\rho}$ is a chain in $f^{-1}(\mathrm{Bd}(\Delta(\rho)))$. By inductive hypothesis, and since $\partial\partial\zeta_{\rho}=0$, there exists a chain $\xi_{\rho}$ in $f^{-1}(\mathrm{Bd}(\Delta(\rho)))$ with $\partial\xi_{\rho}=\partial\zeta_{\rho}$. Since $f^{-1}(\Delta(\rho))$ is $n$-acyclic we get a chain $\eta_{\rho}$ in $f^{-1}(\Delta(\rho))$ with $\eta_{\rho}-\xi_{\rho}=\partial\eta_{\rho}$. We have that 
\[\zeta-\partial(\sum_{\rho}\xi_{\rho})=\sum_{\rho}\xi_{\rho}+\zeta^-.\]
Since $\sum_{\rho}\xi_{\rho}+\zeta^{-}$ is a $(\leq n)$-chain in $f^{-1}(\mathrm{Sk}^{l-1}(A))$ with $\partial(\sum_{\rho}\xi_{\rho}+\zeta^{-})=\partial \zeta -\partial\partial(\sum_{\rho}\xi_{\rho})=0$  we have, by inductive hypothesis, a chain $\eta^{-}$ with $\partial\eta^{-}= \sum_{\rho}\xi_{\rho}+\zeta^{-}$. Set $\eta=\sum_{\rho}\eta_{\rho}+\eta^{-}$.
\end{proof}

\begin{lemma}\label{L:aux2}
Let $f\colon B\to A$ be simplicial such that $f^{-1}(\Delta(\sigma))$ is $l$-acyclic  for every $\sigma \in A$. Let $\orient{\sigma}$ and $\orient{\tau}$ be oriented faces of $B$ with $f_{\#}(\orient{\sigma})=f_{\#}(\orient{\tau})=\orient{\rho}$. If $\orient{\sigma}$, $\orient{\tau}$, $\orient{\rho}$ have dimension $l$ and  
$f_\#(\orient{\sigma}) + f_\#(\orient{\tau})=0$, then there is an $l+1$-chain $\epsilon$ in $f^{-1}(\Delta(\rho))$ and an $l$-chain $\gamma$ in $ f^{-1}(\mathrm{Bd}(\Delta(\rho)))$ so that
\[
\orient{\sigma} + \orient{\tau} = \partial \epsilon + \gamma.
\]
\end{lemma}
\begin{proof}[Sketch of Proof.]
The proof is by induction on $l$.  By \eqref{E:boun}, we have that 
\[
\partial \orient{\sigma} = \sum_{\nu} \orient{\sigma}_{\nu}\;\hbox{ and }\;\partial \orient{\tau} = \sum_{\nu} \orient{\tau}_{\nu}, 
\]
where $\nu$ varies over all $(l-1)$-dimensional faces with $\nu\subseteq\rho$ and   $f(\sigma_{\nu}) = f(\tau_{\nu})=\nu$. It follows that $
f_\#(\sigma_{\nu}) + f_\#(\tau_{\nu}) =0$ and therefore, by inductive assumption, we have that
$\sigma_{\nu}+\tau_{\nu}=\partial\epsilon_{\nu}+ \gamma_{\nu}$, for an $l$-chain $\epsilon_{\nu}$ in $f^{-1}(\Delta(\nu))$ and  an $l-1$-chain $\gamma_{\nu}$ in $f^{-1}(\mathrm{Bd}(\Delta(\nu)))$. One can check now that Lemma~\ref{L:aux2} applies to the chain $\sum_{\nu} \gamma_{\nu}$, producing an $l$-chain $\gamma$ in $f^{-1}(\mathrm{Bd}(\Delta(\rho)))$ with $\sum_{\nu} \gamma_{\nu} =\partial\gamma$. Since
$\partial (\orient{\sigma}+\orient{\tau}- (\sum_{\nu}\epsilon_{\nu} +\gamma))=0$ and
 $f^{-1}(\Delta(\rho))$ is $l$-acyclic, there exists an $l+1$-chain $\epsilon$ in $f^{-1}(\Delta(\rho))$ such that 
\[
\orient{\sigma}+\orient{\tau}- (\sum_{\nu}\epsilon_{\nu} +\gamma) = \partial\epsilon. 
\]
It follows that $\orient{\sigma}+\orient{\tau} = \partial\epsilon + (\sum_{\nu}\epsilon_{\nu} +\gamma)$,  where $\sum_{\nu}\epsilon_{\nu} +\gamma$ is an $l$-chain in $f^{-1}(\mathrm{Bd}(\Delta(\rho)))$, as required.
\end{proof}

\begin{proof}[Proof Sketch of Lemma \ref{L:l-n-acyclicity}.]
Assume without loss of generality that $l\leq n$ and notice that (2) implies that for every $\tau\in \mathrm{Sk}^l(A)$, there is $\sigma\in A$, with $f\sigma=\tau$. 
Let $\zeta_B$ be a $(\leq n-1)$-chain in $B$ with $\partial \zeta_B$. We will find a chain $\eta_B$ with $\partial \eta_B=\zeta_B$.
\begin{claim}
We can assume without loss of generality  that $f_{\#}(\zeta_B)=0$.
\end{claim}
\begin{proof}[Proof of claim]
Set $\zeta=f_{\#}(\zeta_B)$. Since $A$ is $(l-1)$-acyclic, we find a $(\leq l)$-chain $\eta$ in $A$ with $\partial \eta=\zeta$. Set $\eta=\sum_i\orient{\tau}_i$.  Since $\tau_i\in\mathrm{Sk}^l(A)$, we can find a chain $\eta'=\sum_i\orient{\sigma}_i$ in $B$, with $\mathrm{dim}(\sigma_i)=\mathrm{dim}(\tau_i)$ and $f_{\#}(\orient{\sigma}_i)=\orient{\tau}_i$. One can now replace $\zeta_B$ with $\zeta_B-\partial\eta'$ which satisfies all the desired properties. Moreover, if $\zeta_B-\partial\eta'=\partial \eta''$ for some cycle $\eta''$, then $\zeta_B=\partial(\eta'+\eta'')$. 
\end{proof}
By Lemma~\ref{L:aux1} we can further assume that $\zeta_B$ is in fact a $(\leq l-1)$-chain. As a consequence, 
\[
\zeta_B =  \sum_i (\orient{\sigma}_i+\orient{\tau}_i) + \zeta', 
\]
where $f_\#(\orient{\sigma}_i)+f_\#(\orient{\tau}_i)=0$,  $\mathrm{dim}(f\sigma_i)=\mathrm{dim}(\sigma_i)=\mathrm{dim}(\tau_i)=\mathrm{dim}(f\tau_i)=l-1$, and 
$\zeta'$ is an $(\leq l-2)$-chain.  Let $\rho_i=f\sigma_i=f\tau_i$. By 
 Lemma~\ref{L:aux2}, for each $i$, there is a chain $\epsilon_i$ and a chain $\gamma_i$ in $ f^{-1}(\mathrm{Bd}(\Delta(\rho_i)))$ such that $\orient{\sigma}_i+\orient{\tau}_i = \partial\epsilon_i + \gamma_i$. 
Thus, 
\[\zeta_B = \partial (\sum_i \epsilon_i) + (\sum_i \gamma_i + \zeta').\]   
One checks now that $f_{\#}(\sum_i \gamma_i + \zeta')$ is a chain in $\mathrm{Sk}^{l-1}(A)$ and the above equation implies that $\partial(\sum_i \gamma_i + \zeta')=0$. By inductive assumption we can find $\eta$ with $\partial\eta =(\sum_i \gamma_i + \zeta')$ and set $\eta_B=(\sum_i \epsilon_i)+\eta$ to be the required chain.
\end{proof}

As in Section \ref{S: Menger curve}, we can now construct generic sequences for $\mathcal{C}_n$ and $\widetilde{\mathcal{C}}_n$ whose inverse limits we denote by $\mathbb{M}^n$ and $\widetilde{\mathbb{M}}^n$ respectively. Both $\mathbb{M}^n$ and $\widetilde{\mathbb{M}}^n$ are compact  $n$-dimensional simplicial complexes and as in Theorem \ref{T: Menger is Menger} it is easy to see that the relation $R$, where $xRy$ iff there is a face $\sigma$ with $x,y\in\sigma$, is an equivalence relation. We let $|\mathbb{M}^n|=\mathbb{M}^n/R$ and $\widetilde{\mathbb{M}}^n=\widetilde{\mathbb{M}}^n/R$. It follows that $|\mathbb{M}^0|$ and $|\widetilde{\mathbb{M}}^0|$ are both homeomorphic to the Cantor space $2^{\mathbb{N}}$; both $|\mathbb{M}^1|$ and $|\widetilde{\mathbb{M}}^1|$ are homeomorphic to the Menger curve $|\mathbb{M}|$; and as in Theorem \ref{T: Menger is Menger} one can see that both $|\mathbb{M}^n|$ and $\widetilde{\mathbb{M}}^n$ are Peano continua. While one expects $|\mathbb{M}^n|$ to be the usual Menger compactum of dimension $n$ (see \cite{Be}), we observe that for $n>1$, the complex $\widetilde{\mathbb{M}}^n$ admits quotients $A\in \widetilde{\mathcal{C}}_n$ which are $(n-1)$-acyclic but not $(n-1)$-connected. To the best of our knowledge these ``homology Menger spaces'', and for $n=\infty$ this  ``homology Hilbert cube,'' have not appeared elsewhere in the literature.

\end{document}